\documentclass{amsart}

\usepackage{amssymb}

\usepackage{amsfonts}
\usepackage{mathrsfs}
\usepackage{bbm}

\usepackage[dvipdfmx, dvipsnames, svgnames, x11names]{xcolor} %

\theoremstyle{plain}
    \newtheorem{theorem}{Theorem}[section]
    \newtheorem{lemma}[theorem]{Lemma}
    \newtheorem{proposition}[theorem]{Proposition}
    \newtheorem{corollary}[theorem]{Corollary}
\theoremstyle{definition}
    \newtheorem{definition}{Definition}[section]
    \newtheorem{remark}{Remark}[section]
    \newtheorem*{acknowledgement}{Acknowledgement}
\theoremstyle{remark}
    
\numberwithin{equation}{section}

\newcommand{\cleq}{\lesssim}

\newcommand{\ceq}{\approx} 
\def\norm#1{\left\Vert #1 \right\Vert} 

\newcommand{\R}{\mathbb{R}}

\DeclareMathOperator{\re}{Re}

\begin{document}

\title[Scattering for NLW with scale-invariant damping]{Scattering and asymptotic order for the 
wave equations with the scale-invariant damping and mass}
\author[T. Inui]{Takahisa Inui}
\address{Department of Mathematics, Graduate School of Science, Osaka University, Toyonaka, Osaka 560-0043, Japan}
\email{inui@math.sci.osaka-u.ac.jp}
\author[H. Mizutani]{Haruya Mizutani}
\address{Department of Mathematics, Graduate School of Science, Osaka University, Toyonaka, Osaka 560-0043, Japan}
\email{haruya@math.sci.osaka-u.ac.jp}
\date{\today}
\maketitle

\begin{abstract}
We consider the linear wave equation with the time-dependent scale-invariant damping and mass. We also treat the corresponding equation with the energy critical nonlinearity. Our aim is to show that the solution scatters to a modified linear wave solution and to obtain its asymptotic order. 
\end{abstract}

\tableofcontents

\section{Introduction}

\subsection{Motivation}
We consider two equations in the present paper. One is the linear wave equation with the scale-invariant damping and mass.
\begin{align}
\label{DW}
\tag{DW}
	\begin{cases}
	\displaystyle \partial_t^2 u_l - \Delta u_l +\frac{\mu_1}{1+t} \partial_t u_l + \frac{\mu_2}{(1+t)^2} u_l =0, & (t,x) \in (0,\infty) \times \mathbb{R}^d,
	\\
	(u_l(0), \partial_t u_l(0)) = (u_{l,0}, u_{l,1}), &  x \in \mathbb{R}^d,
	\end{cases}
\end{align}
where $\mu_1,\mu_2 \in \mathbb{R}$, $d\in \mathbb{N}$, and $(u_{l,0}, u_{l,1}) \in H^1(\mathbb{R}^d)\times L^2(\mathbb{R}^d)$. 
The other is the corresponding wave equation with the energy critical nonlinearity. 
\begin{align}
\label{NLDW}
\tag{NLDW}
	\begin{cases}
	\displaystyle \partial_t^2 u -\Delta u +\frac{\mu_1}{1+t} \partial_t u + \frac{\mu_2}{(1+t)^2} u = \lambda |u|^{\frac{4}{d-2}}u,
	& (t,x) \in (0,T) \times \mathbb{R}^d,
	\\
	(u(0),\partial_t u(0))=(u_0, u_1), &x \in \mathbb{R}^d,
	\end{cases}
\end{align}
where $\mu_1  \in [0,\infty)$, $\mu_2 \in \mathbb{R}$, $d\geq 3$, $\lambda =\pm 1$, and $(u_0, u_1) \in H^1(\mathbb{R}^d) \times L^2(\mathbb{R}^d)$.

First, we consider the linear equation. 
Before considering our linear equation, we recall the classification by Wirth \cite{Wir04,Wir06,Wir07_1} of the following wave equation with time-dependent damping.
\begin{align}
\label{eq1.0}
	\partial_t^2 w_l -\Delta w_l +\frac{b_0}{(1+t)^{\beta}} \partial_t w_l = 0,
	\quad (t,x) \in (0,T) \times \mathbb{R}^d,
\end{align}
where $b_0>0$ and $\beta \in \mathbb{R}$. 
When $\beta <-1$, the equation is called overdamping. In this case, the solution does not decay to zero when $t \to \infty$. 
When $-1 \leq \beta <1$, the damping term is called effective. In this case for any $b_0$,  the solution behaves like one of the heat equation obtained formally by taking $\partial_t^2w_{l}\equiv 0$. When $\beta >1$, it is known that, for any $b_0$, the solution scatters to a solution of the free wave equation. Then, we say that we have scattering. When $\beta =1$, then the equation is invariant under the scaling
\begin{align*}
	w_l^{\sigma}(t,x):= w_l( \sigma (1+t)-1,\sigma x).
\end{align*}
Namely, $w_l^{\sigma}$ is a solution provided that $w_l$ is a solution. Thus, this case is called scale-invariant. 
\begin{table}[htb]
  \begin{tabular}{|l|l|} \hline
    $\beta\in (-\infty,-1)$ & overdamping  \\ \hline
    $\beta\in [-1,1)$ & effective  \\ \hline
    $\beta =1$ & scale-invariant  \\ \hline
    $\beta\in (1,\infty)$ & scattering  \\ \hline
  \end{tabular}
\end{table}

In the scale-invariant case, Wirth \cite{Wir04} showed that the solution satisfies the following $L^p$-$L^q$ estimates.
\begin{align*}
	\|w_l(t)\|_{L^q}
	\lesssim 
	\begin{cases}
	(1+t)^{\max\{ -\frac{d-1}{2}\left(\frac{1}{p}-\frac{1}{q}\right)-\frac{b_0}{2}, -d\left(\frac{1}{p}-\frac{1}{q}\right)+1-b_0\}}, 
	&\text{ if } b_0\in (0,1),
	\\
	(1+t)^{\max\{ -\frac{d-1}{2}\left(\frac{1}{p}-\frac{1}{q}\right)-\frac{b_0}{2},-d\left(\frac{1}{p}-\frac{1}{q}\right)\}},
	&\text{ if } b_0\in (1,\infty),
	\end{cases} 
\end{align*}
where $1<p\le 2$, $1/p+1/q=1$, the implicit constant depends on $\|w_{l,0}\|_{H^s_p}+\|w_{l,1}\|_{H_p^{s-1}}$, and  $s=d(1/p-1/q)$.
This shows that, if $b_0$ is sufficiently large, then the solution behaves like that of the corresponding heat equation 
\begin{align*}
	\frac{b_0}{1+t} \partial_t h-\Delta h = 0,
\end{align*}
whose decay order is given by $-n(1/p-1/q)$, and that, if $b_0$ is sufficiently small, then the solution behaves like that of the wave equation. 
Therefore, the scale-invariant case is critical in the sense of the global behavior of the solutions and the constant $b_0$ plays an important role to determine the global behavior of the solutions unlike the effective and scattering case.

It is known  by Wirth \cite{Wir06,Wir07_2} that, if $b_0 \in (0,1) \cup (1,2)$, then the solution $w_l$ satisfies that there exists a function $\overrightarrow{w_{+}} \in \dot{H}^1(\mathbb{R}^d) \times L^2(\mathbb{R}^d)$ such that
\begin{align}
\label{eq1.1}
	\lim_{t \to \infty} (1+t)^{\frac{b_0}{2}} \| \overrightarrow{w_l}(t) -  (1+t)^{-\frac{b_0}{2}}\mathcal{W}(t) \overrightarrow{w_+} \|_{\dot{H}^1 \times L^2} =0,
\end{align}
where $\mathcal{W}(t)$ is a solution map of the free wave equation and $\overrightarrow{w_l}(t)=(w_l(t),\partial_t w_l(t))$. 
Note that $\overrightarrow{w_+} \not \equiv 0$ for non-trivial initial data (see \cite[Theorem 3.1]{Wir07_2}). 
This means that the solution $\overrightarrow{w_l}$ behaves like the modified linear wave solution $(1+t)^{-\frac{b_0}{2}}\mathcal{W}(t) \overrightarrow{w_+}$ at infinite time and its asymptotic order is $(1+t)^{-\frac{b_0}{2}}$. 
This result can be also interpreted as follows. By the Liouville transform $\varphi_l := (1+t)^{\frac{b_0}{2}} w_l$, $\varphi_l$ satisfies the linear Klein-Gordon equation with time-dependent mass
\begin{align*}
	\begin{cases}
	\displaystyle \partial_t^2 \varphi_l - \Delta \varphi_l +\frac{b_0(2-b_0)}{4(1+t)^2} \varphi_l = 0, & (t,x) \in (0,\infty) \times \mathbb{R}^d,
	\\
	(\varphi_l(0),\partial_t \varphi_l(0))=(\varphi_{l,0},\varphi_{l,1}):=(w_l(0),\frac{b_0}{2} w_l(0)+\partial_t w_l(0)), &x \in \mathbb{R}^d.
	\end{cases}
\end{align*}
Therefore, \eqref{eq1.1} means that $\varphi_l$ behaves like the solution to the free wave equation. 
In this viewpoint, the asymptotic order is not clear since we already use the transformation $\varphi_l = (1+t)^{\frac{b_0}{2}} w_l$. However, in the special case of $b_0=2$, $\varphi_l$ is nothing but the solution of the free wave equation. Therefore, for any $\alpha>0$, we have $\lim_{t \to \infty} (1+t)^{\alpha}\| \overrightarrow{\varphi_l}(t) -  \mathcal{W}(t)(\varphi_{l,0},\varphi_{l,1}) \|_{\dot{H}^1 \times L^2} =0$. Therefore, turning back to \eqref{eq1.0} by retransform, we get the asymptotic order $(1+t)^{-\alpha-\frac{b_0}{2}}$ for $w_l$. 
This observation makes us expect to get the asymptotic order for \eqref{DW} with more general $\mu_1$ and $\mu_2$. In the present paper, we discuss the asymptotic order for \eqref{DW}.

Next, we are also interested in the nonlinear problem. 
Recently, the nonlinear wave equation with the scale-invariant damping
\begin{align*}
	\partial_t^2 w -\Delta w +\frac{b_0}{(1+t)^{\beta}} \partial_t w = \lambda |w|^{p},
	\quad (t,x) \in (0,T) \times \mathbb{R}^d,
\end{align*}
where $b_0>0$, $\beta \in \mathbb{R}$, and $p>1$, has been well studied from the viewpoint of a critical exponent $p_*$. 
%
The critical exponent $p_*$ is the boundary between  small data global existence and small data blow-up. In the overdamping case, $p_*=\infty$, that is, the small data global existence holds for all $p>1$. See \cite{IkWa20}. In the effective case, $p_*$ is the Fujita exponent $p_{F}=1+2/d$, which is the critical exponent for the heat equation with the power type nonlinearity. See \cite{Nis11,LNZ12,DLR13,IkWa15,FIW16a,LaZh17a,Wak17,IkIn19} and references therein. In the scale-invariant case, the critical exponent is still not clear in general. However, the Fujita exponent $p_F$ is no longer critical. See \cite{Wak14,Dab15,DaLu15,DLR15,LTW17a,NPR17,IkSo18,PaRe18,PaRe19,LST20} and references therein for the concrete information. In the scattering case, the critical exponent is the Strauss exponent, which is the critical exponent for the free wave equation. See \cite{LaTa18,LST19} and references therein.

Turn back to our nonlinear equation. The power $1+4/(d-2)$ of our energy critical nonlinearity seems to be much larger than its critical exponent. Thus, our first aim is to show small data global existence. And the second aim is to obtain the asymptotic behavior of the nonlinear solution.

The first author \cite{Inu18} considered the following type equation. 
\begin{align}
\label{eq1.4}
	\partial_t^2 w -\Delta w +\frac{2}{1+t} \partial_t w = \lambda |w|^{\frac{4}{d-2}}w,
	\quad (t,x) \in (0,T) \times \mathbb{R}^d.
\end{align}
In the similar way to the linear problem, one can transform this equation by the Liouville transform $\varphi = (1+t)^{\frac{\mu_1}{2}} w$ into the nonlinear wave equation 
\begin{align}
\label{eq1.5}
	\partial_t^2 \varphi - \Delta \varphi =  \frac{\lambda}{(1+t)^{\frac{4}{d-2}}} |\varphi|^{\frac{4}{d-2}}\varphi, 
	\quad (t,x) \in (0,T) \times \mathbb{R}^d.
\end{align}
It was proved by \cite{Inu18} that  if $d \geq 3$, $\varphi$ is a global solution of \eqref{eq1.5}, and $\varphi$ satisfies $\| \varphi \|_{L_{t,x}^{2(d+1)/(d-2)}([0,\infty))}<\infty$, then there exists $\overrightarrow{\varphi_+} \in \dot{H}^1(\R^d) \times L^2(\R^d)$ such that
$\lim_{t \to \infty} (1+t)^{4/(d-2)} \|\overrightarrow{\varphi}(t) -\mathcal{W}(t) \overrightarrow{\varphi_+}\|_{\dot{H}^{1} \times L^2} =0$, where $\overrightarrow{\varphi}=(\varphi, \partial_t \varphi)$. This result means that the solution to \eqref{eq1.5} scatters to a solution of the free wave solution and its asymptotic order is $(1+t)^{-4/(d-2)}$. 
The order comes from the decay $(1+t)^{-4/(d-2)}$ in front of the nonlinearity of \eqref{eq1.5}. 
Turning back to \eqref{eq1.4} by retransform $w= (1+t)^{-\frac{\mu_1}{2}} \varphi$, we can obtain the small data global existence, scattering results, and asymptotic order for \eqref{eq1.4}. Our aim in the present paper is to get the asymptotic behavior and its asymptotic order for \eqref{NLDW} with the more general $\mu_1$ and $\mu_2$.

\subsection{Main results}


\subsubsection{Linear problem} 
By the Liouville transfrom $v_l := (1+t)^{\frac{\mu_1}{2}} u_l$, we transform \eqref{DW} into  the following Klein-Gordon equation with scale-invariant mass.
\begin{align}
\label{KG}
\tag{KG}
	\begin{cases}
	\displaystyle \partial_t^2 v_l - \Delta v_l +\frac{\mu}{(1+t)^2}v_l =0, & (t, x) \in (0,\infty) \times \mathbb{R}^d,
	\\
	(v_l(0), \partial_t v_l(0)) = (v_{l,0}, v_{l,1}), &  x \in \mathbb{R}^d,
	\end{cases}
\end{align}
where we set $(v_{l,0},v_{l,1}):=(u_{l,0}, \mu_1u_{l,0}/2+u_{l,1})$ and 
\begin{align*}
	\mu:=\frac{\mu_1(2-\mu_1)}{4} +\mu_2.
\end{align*}
We set $\overrightarrow{v_l}(t)=(v_l(t),\partial_t v_l(t))$. 
Let $\nu:=\frac{1}{2}\sqrt{1-4\mu}$ when $\mu \leq 1/4$ and $\nu:=\frac{i}{2}\sqrt{4\mu-1}$ when $\mu > 1/4$.  
In what follows, $\mathcal{W}(t)$ denotes the solution map of the free wave equation (see Section \ref{sec2.1} below).

Then, we get the following scattering and the asymptotic order. 

\begin{theorem}
\label{thm1.0}
Let $d\in \mathbb{N}$ and $\mu \geq 0$. Then, there exists $\overrightarrow{v_{+}} \in \dot{H}^1(\mathbb{R}^d) \times L^2(\mathbb{R}^d)$ such that
\begin{align*}
	\norm{\overrightarrow{v_l}(t) - \mathcal{W}(t)\overrightarrow{v_{+}}}_{\dot{H}^1 \times L^2}
	\cleq \mu (\norm{v_{l,0}}_{L^2}+\norm{v_{l,1}}_{L^2})
	\begin{cases}
	(1+t)^{-\frac{1}{2} + \re\nu}, & \mu \neq 1/4,
	\\
	(1+t)^{-\frac{1}{2}}(1+ \log(1+t)), & \mu =1/4,
	\end{cases}
\end{align*}
where the implicit constant does not depend on time and the initial data $(v_{l,0},v_{l,1})$. 
\end{theorem}

\begin{remark}
\ 
\begin{enumerate}
\item We note that $\re \nu <1/2$ when $\mu>0$ and thus the right hand side decays. 
If $\mu=0$, then the right hand side is identically zero. 
\item If $\mu_2=0$, i.e. no mass term, then the assumption $\mu\geq 0$ implies $\mu_1 \in [0,2]$. The case of $\mu_1>2$ and $\mu_2=0$ is not treated in the paper. 
\item It is known that $\overrightarrow{v_+}$ is not identically zero when $0\leq \mu <1/4$. Indeed, it is trivial when $\mu=0$.  When $\mu\in (0,1/4)$, setting $u_l =(1+t)^{-\frac{\mu_1}{2}}v_l=(1+t)^{-\frac{\mu_1}{2}+\frac{1}{2} + \nu}w_l$, we get
\begin{align*}
	\partial_t^2 w_l -\Delta w_l +\frac{1+2\nu}{1+t} \partial_t w_l = 0,
	\quad (t,x) \in (0,T) \times \mathbb{R}^d,
\end{align*}
where we note that $\nu \in (0,1/2)$ when $\mu\in (0,1/4)$. By the result of Wirth \cite{Wir07_2}, we have $\| \overrightarrow{w_l} (t) \|_{\dot{H}^1 \times L^2} \ceq (1+t)^{-1/2-\nu}$. This implies that $\| \overrightarrow{v_l} (t) \|_{\dot{H}^1 \times L^2} \ceq 1$. Therefore, $\overrightarrow{v_+} \neq 0$ in Theorem \ref{thm1.0} when $\mu\in (0,1/4)$. 
\end{enumerate}

\end{remark}


By retransforming $u_l=(1+t)^{-\mu_1/2}v_l$, we have the following result for \eqref{DW}. This is an improvement of \eqref{eq1.1} above obtained by \cite{Wir06,Wir07_2} from the viewpoint of asymptotic order. 
\begin{corollary}
\label{cor1.0}
Let $d\in \mathbb{N}$ and  $\mu \geq 0$ and $u_l$ be the solution of \eqref{DW}. Then, there exists $\overrightarrow{v_{+}} \in \dot{H}^1(\mathbb{R}^d) \times L^2(\mathbb{R}^d)$ such that
\begin{align*}
	&\norm{u_l(t) -(1+t)^{-\frac{\mu_1}{2}} (\mathcal{W}(t)\overrightarrow{v_{+}} )_1}_{\dot{H}^1}
	\\
	&\cleq \mu 
	(\norm{u_{l,0}}_{L^2}+\norm{u_{l,1}}_{L^2})
	\begin{cases}
	(1+t)^{-\frac{1}{2} + \re\nu - \frac{\mu_1}{2}}, & \mu \neq 1/4,
	\\
	(1+t)^{-\frac{1}{2} - \frac{\mu_1}{2}}(1+\log(1+t)), & \mu =1/4,
	\end{cases}
\end{align*}
and 
\begin{align*}
	&\norm{\partial_t u_l(t) -(1+t)^{-\frac{\mu_1}{2}} (\mathcal{W}(t)\overrightarrow{v_{+}})_2 }_{L^2}
	\\
	&\cleq (\mu +|\mu_1|)
	(\norm{u_{l,0}}_{L^2}+\norm{u_{l,1}}_{L^2})
	\begin{cases}
	(1+t)^{-\frac{1}{2} + \re\nu - \frac{\mu_1}{2}}, & \mu \neq 1/4,
	\\
	(1+t)^{-\frac{1}{2} - \frac{\mu_1}{2}} (1+\log(1+t)), & \mu =1/4,
	\end{cases}
\end{align*}
where $( \mathcal{W}(t)\overrightarrow{v_{+}})_{j}$ denotes the $j$-th component of $\mathcal{W}(t)\overrightarrow{v_{+}}$. 
\end{corollary}

\begin{remark}
In the sequel paper \cite{InMi20_2}, we will consider the case with rapidly decaying initial data for which the asymptotic order can be improved. 
\end{remark}





\subsubsection{Nonlinear problem}
We consider the nonlinear equation \eqref{NLDW}. By the Liouville transform $v = (1+t)^{\frac{\mu_1}{2}} u$, 
$v$ satisfies the following energy critical nonlinear Klein-Gordon equation with scale-invariant mass.
\begin{align}
\label{NLKG}
\tag{NLKG}
	\begin{cases}
	\displaystyle \partial_t^2 v - \Delta v +\frac{\mu  }{(1+t)^2} v =  \frac{\lambda}{(1+t)^{\frac{\mu_1}{2}(p_1-1)}} |v|^{p_1-1}v, & (t,x) \in (0,T) \times \mathbb{R}^d,
	\\
	(v(0),\partial_t v(0))=(v_0,v_1), &x \in \mathbb{R}^d,
	\end{cases}
\end{align}
where we set $p_1:=1+ 4/(d-2)$ and $(v_0,v_1):=(u_0, \mu_1u_0/2+u_1)$. Here, $\mu$ is same as in \eqref{KG}.

Then, we obtain the following local well-posedness. 

\begin{theorem}
\label{thm1.1}
We assume $\mu_1\geq 0$ and $\mu \geq 0$. 
Let $3\leq d \leq 6$, $T\in (0,\infty]$, and $(v_0,v_1) \in H^1 (\mathbb{R}^d) \times L^2(\mathbb{R}^d)$. Then, there exists $\delta>0$ such that 
if $\| \mathcal{E}_{0}(t,0) v_0 + \mathcal{E}_{1}(t,0) v_1\|_{L^{p_1}([0,T): L_x^{2p_1})} \leq \delta$, then there exists a unique solution to \eqref{NLKG} on $[0,T)$ in the sense of Definition \ref{def} below. 
\end{theorem}

By this local well-posedness, we can show the small data global existence. 

\begin{corollary}
\label{cor1.4}
We assume $\mu_1\geq 0$ and $\mu \geq 0$. 
Let $3\leq d \leq 6$. 
If $\| v_0 \|_{H^1}+\| v_1 \|_{L^2}$ is sufficiently small, then there exists a unique global solution and it satisfies $\| v \|_{L^{p_1}([0,\infty): L_x^{2p_1})} \cleq \| v_0 \|_{H^1}+\| v_1 \|_{L^2}$. 
\end{corollary}

Moreover, we obtain the scattering result and the asymptotic order as follows. 

\begin{theorem}
\label{thm1.3}
Let $\mu_1 > 0$, $\mu \geq 0$ and $3\leq d \leq 5$. 
If $v$ is a global solution  to \eqref{NLKG} in the sense of Definition \ref{def} and $v$ satisfies $\| v \|_{L^{p_1}([0,\infty): L_x^{2p_1})} <\infty$, then there exists $\overrightarrow{v_{+}} \in \dot{H}^1 \times L^2$ such that
\begin{align*}
	&\norm{\overrightarrow{v}(t) - \mathcal{W}(t)\overrightarrow{v_{+}}}_{\dot{H}^1 \times L^2}
	\\
	&\cleq   
	\begin{cases}
	\mu(1+t)^{\max\{ -\frac{1}{2} + \re\nu, -\frac{2\mu_1}{d-2}\} } + (1+t)^{-\frac{2\mu_1}{d-2}} o_t(1), & \mu \neq 1/4,
	\\
	\mu(1+t)^{\max\{-\frac{1}{2},-\frac{2\mu_1}{d-2}\}}(1+\log(1+t)) +(1+t)^{-\frac{2\mu_1}{d-2}}o_t(1), & \mu = 1/4,
	\end{cases}
\end{align*}
where $o_t(1)=\| v \|_{L^{p_1}([t,\infty): L_x^{2p_1})}$, which goes to $0$ as $t \to \infty$, and $\overrightarrow{v}(t)=(v(t),\partial_t v(t))$.
\end{theorem}

\begin{remark}
\ 
\begin{enumerate}
\item Theorem \ref{thm1.3} means that the solution to \eqref{NLKG} with the finite space-time norm behaves like the free wave equation at infinite time. We also obtain the asymptotic order $\max\{ -\frac{1}{2} + \re\nu, -\frac{2\mu_1}{d-2}\}$. 
As seen in the linear case, the order $-\frac{1}{2} + \re\nu$ comes from the linear part. On the other hand, 
the order $-\frac{2\mu_1}{d-2}$ comes from the decay in front of the nonlinearity in \eqref{NLKG}. 
\item Theorem \ref{thm1.3} covers the first author's result \cite{Inu18}, in which he considered the case $\mu=0$. If $\mu >0$, then we could not get decay term $o_t(1)$, that is, it is not clear whether $\lim_{t \to \infty} (1+t)^{-\max\{ -\frac{1}{2} + \re\nu, -\frac{2\mu_1}{d-2}\} } \norm{\overrightarrow{v}(t) - \mathcal{W}(t)\overrightarrow{v_{+}}}_{\dot{H}^1 \times L^2} =0$ or not.
\item In Theorem \ref{thm1.1}, we assume $d \leq 6$ to use the Strichartz estimate in Proposition \ref{prop2.3}. On the other hand, we exclude $d=6$ in Theorem \ref{thm1.3} since we use the smoothness of the nonlinearity to obtain scattering. See Appendix \ref{appB}. 
\item Our proof is also applicable to the nonlinearity $\lambda |v|^{1+4/(d-2)}$. 
\end{enumerate}
\end{remark}


By retransforming $u=(1+t)^{-\mu_1/2}v$, we have the following result for \eqref{NLDW}. 
\begin{corollary}
\label{cor1.6}
Let $\mu_1> 0$, $\mu \geq 0$ and $3\leq d \leq 5$. If $u$ is a global solution  to \eqref{NLDW} and $u$ satisfies $\| (1+t)^{\mu_1/2}u \|_{L^{p_1}([0,\infty): L_x^{2p_1})} <\infty$, then there exists $\overrightarrow{v_{+}} \in \dot{H}^1 \times L^2$ such that
\begin{align*}
	&\norm{u(t) -(1+t)^{-\frac{\mu_1}{2}} (\mathcal{W}(t)\overrightarrow{v_{+}} )_1}_{\dot{H}^1}
	\\
	&\cleq  
	\begin{cases}
	\mu(1+t)^{\max\{ -\frac{1}{2} + \re\nu, -\frac{2\mu_1}{d-2}\}-\frac{\mu_1}{2} } + (1+t)^{-\frac{2\mu_1}{d-2}-\frac{\mu_1}{2}} o_t(1), & \mu \neq 1/4,
	\\
	\mu(1+t)^{\max\{-\frac{1}{2},-\frac{2\mu_1}{d-2}\}-\frac{\mu_1}{2}}(1+\log(1+t)) +(1+t)^{-\frac{2\mu_1}{d-2}-\frac{\mu_1}{2}}o_t(1), & \mu = 1/4,
	\end{cases}
\end{align*}
and 
\begin{align*}
	&\norm{\partial_t u(t) -(1+t)^{-\frac{\mu_1}{2}} (\mathcal{W}(t)\overrightarrow{v_{+}})_2 }_{L^2}
	\\
	&\cleq 
	\begin{cases}
	(\mu+\mu_1)(1+t)^{\max\{ -\frac{1}{2} + \re\nu, -\frac{2\mu_1}{d-2}\}-\frac{\mu_1}{2} } + (1+t)^{-\frac{2\mu_1}{d-2}-\frac{\mu_1}{2}} o_t(1), & \mu \neq 1/4,
	\\
	(\mu+\mu_1)(1+t)^{\max\{-\frac{1}{2},-\frac{2\mu_1}{d-2}\}-\frac{\mu_1}{2}}(1+\log(1+t)) +(1+t)^{-\frac{2\mu_1}{d-2}-\frac{\mu_1}{2}}o_t(1), & \mu = 1/4,
	\end{cases}
\end{align*}
where $o_t(1)=\| (1+t)^{\mu_1/2}u \|_{L^{p_1}([t,\infty): L_x^{2p_1})}$, which goes to $0$ as $t \to \infty$
\end{corollary}

\subsection{Idea of the proofs}

Our proof for the linear equation \eqref{KG} is based on the two expressions of the equation \eqref{KG}. One is to use the linear propagator $\mathcal{E}$. The other is to regard it as the wave equation with the inhomogeneous term $\mu(1+t)^{-2}v_l$. It is known that the $L^2$-norm of the free wave equation may grow like $1+t$. Then,  the $L^2$-norm of the inhomogeneous term $\mu(1+t)^{-2}v_l$ is not integrable for $t$. This is the main difficulty. To overcome this difficulty, we use the expression by the solution map $\mathcal{E}$. By using this formulation, B\"{o}hme and Reissig \cite{BoRe12} showed $v_l$ satisfies 
\begin{align*}
	\| v_l(t) \|_{L^2} \cleq 
	\begin{cases}
	(1+t)^{\frac{1}{2}+\re \nu}, & \mu \neq \frac{1}{4},
	\\
	(1+t)^{\frac{1}{2}} (1+\log(1+t)),& \mu =\frac{1}{4}.
	\end{cases}
\end{align*}
The order $\frac{1}{2}+\re \nu$ is less than $1$. Therefore, in fact,  $\mu(1+t)^{-2}v_l$ is integrable for $t$. This is the key point in the present  paper. 

Next we show the idea of the nonlinear problem. In the previous paper \cite{Inu18} by the first author, who considered the case $\mu=0$, the linear part is the free wave equation and thus the space-time estimates for the free wave equation, which is called the Strichartz estimates, are applicable. However, the Strichartz estimates for the free wave equation  are not sufficient  to show the small data global existence for \eqref{NLKG}. Instead, we will prove the Strichartz estimate for the propagator $\mathcal{E}$ of \eqref{KG} by using the $L^2$ estimate for $v_l$ mentioned above and apply it to show the small data global existence. After that, we discuss the scattering and the asymptotic order by using the Strichartz estimates for the free wave equation. 

\section{Preliminaries}

\subsection{Expression of the linear equation}
\label{sec2.1}

We consider the linear Klein-Gordon equation with time-dependent mass with the initial data given at $t=t_0\geq 0$.
\begin{align}
\label{eq2.1}
	\begin{cases}
	\displaystyle\partial_t^2 v_l - \Delta v_l +\frac{\mu}{(1+t)^2}v_l =0, & (t, x) \in (t_0,\infty) \times \mathbb{R}^d,
	\\
	(v_l(t_0), \partial_t v_l(t_0)) = (v_{l,0}, v_{l,1}), &  x \in \mathbb{R}^d.
	\end{cases}
\end{align}
For the expression of the solution, we do not need to assume $\mu \geq 0$. 
By the Fourier transform, we have
\begin{align}
\label{eq2.2}
	\frac{d^2}{dt^2} \widehat{v_l} +|\xi|^2 \widehat{v_l} +\frac{\mu}{(1+t)^2} \widehat{v_l} =0,
\end{align}
where $\widehat{f}$ denotes the spatial Fourier transform of $f$. 
Setting $\tau :=(1+t)|\xi|$ and $\widehat{v_l} (t) = \tau^{1/2} V_l(\tau)$, we obtain the Bessel differential equation
\begin{align*}
	\frac{d^2 V_l  }{d\tau^2} + \frac{1}{\tau} \frac{dV_l }{d\tau} + \left( 1- \frac{\nu^2}{\tau^2}\right)V_l =0,
\end{align*}
where we recall $\nu:=\frac{1}{2}\sqrt{1-4\mu}$ when $\mu \leq 1/4$ and $\nu:=\frac{i}{2}\sqrt{4\mu-1}$ when $\mu > 1/4$.
Then the fundamental solutions are given by the Bessel function (of the first kind) $J_{\nu}(\tau)$ and the Neumann function (the Bessel function of the second type) $Y_{\nu}(\tau)$, where the Bessel function $J_{\nu}$ and $Y_{\nu}$ are defined by 
\begin{align*}
	J_{\nu} (\tau) &:= \sum_{k=0}^{\infty} \frac{(-1)^{k}}{k! \Gamma(\nu+k+1)} \left( \frac{\tau}{2} \right)^{\nu+2k},
	\\
	Y_{\nu}(\tau)&:= \frac{J_{\nu}(\tau)\cos (\nu \pi) - J_{-\nu} (\tau)}{\sin (\nu \pi)} \text{ for } \nu \in  \mathbb{C} \setminus \mathbb{Z}.
\end{align*}
For a non-negative integer $n$, $Y_{n}$ is defined by $\lim_{\nu \to n} Y_{\nu}$ and for negative integer $-n$, $Y_{-n}:=(-1)^nY_{n}$. 
Therefore, the fundamental solutions of \eqref{eq2.2} are given by
\begin{align*}
	e_{+}(t,\xi)&:=((1+t)|\xi|)^{\frac{1}{2}} J_{\nu}((1+t)|\xi|)
	\text{ and }
	e_{-}(t,\xi):=((1+t)|\xi|)^{\frac{1}{2}} Y_{\nu}((1+t)|\xi|) 
\end{align*}

With the initial data $(\widehat{v_l}(t_0), \partial_t \widehat{v_l}(t_0)) = (\widehat{v_{l,0}}, \widehat{v_{l,1}})$, $\widehat{v_l}$ is given by 
\begin{align*}
	\widehat{v_l}(t,\xi) =E_{0}(t,t_0,\xi) \widehat{v_{l,0}}(\xi) +E_{1}(t,t_0,\xi)\widehat{v_{l,1}}(\xi)
\end{align*}
where 
\begin{align*}
	E_{0}(t,t_0,\xi)
	&:=\frac{e_{+}(t,\xi) \dot{e_{-}}(t_0,\xi) - \dot{e_{+}}(t_0,\xi) e_{-}(t,\xi)}{e_{+}(t_0,\xi) \dot{e_{-}}(t_0,\xi) - \dot{e_{+}}(t_0,\xi) e_{-}(t_0,\xi) },
	\\
	E_{1}(t,t_0,\xi)
	&:=\frac{e_{+}(t_0,\xi) e_{-}(t,\xi) - e_{+}(t,\xi) e_{-}(t_0,\xi)}{e_{+}(t_0,\xi) \dot{e_{-}}(t_0,\xi) - \dot{e_{+}}(t_0,\xi) e_{-}(t_0,\xi) }
\end{align*}
and $\dot{e_{+}}$ and $\dot{e_{-}}$ denote the time derivative of $e_{+}$ and $e_{-}$, respectively. 
We also have 
\begin{align*}
	\frac{d\widehat{v_l}}{dt}(t,\xi) =\dot{E_{0}}(t,t_0,\xi) \widehat{v_{l,0}}(\xi) +\dot{E_{1}}(t,t_0,\xi)\widehat{v_{l,1}}(\xi)
\end{align*}
where $\dot{E_0}$ and $\dot{E_1}$ denote the time derivative of $E_0$ and $E_1$, respectively. 
Therefore, the solution $v_l$ of \eqref{eq2.1} is given by 
\begin{align}
\label{eq2.3.0}
	v_l(t) = \mathcal{E}_{0}(t,t_0) v_{l,0} + \mathcal{E}_{1}(t,t_0) v_{l,1}.
\end{align}
where $\mathcal{E}_{i}(t,t_0) = \mathcal{F}^{-1} E_{i}(t,t_0,\xi) \mathcal{F}$ for $i=0,1$ and $\mathcal{F}, \mathcal{F}^{-1}$ are the spatial Fourier transform and its inverse, respectively. 
By the vector expression, we have
\begin{align*}
	\begin{pmatrix}
	v_l \\ \dot{v_l}
	\end{pmatrix}
	= \mathcal{E}(t,t_0)
	\begin{pmatrix}
	v_{l,0} \\ v_{l,1}
	\end{pmatrix}
\end{align*}
where 
\begin{align*}
	\mathcal{E}(t,t_0)=
	\begin{pmatrix}
	 \mathcal{E}_{0}(t,t_0) &  \mathcal{E}_{1}(t,t_0)
	 \\
	 \dot{\mathcal{E}_{0}}(t,t_0) &  \dot{\mathcal{E}_{1}}(t,t_0)
	\end{pmatrix}.
\end{align*}

\begin{remark}
Our expression is based on the Bessel functions. On the other hand, B\"{o}hme and Reissig \cite{BoRe12} give the expression of the solution to \eqref{eq2.1} by the hypergeometric functions of confluent type. Of course, these expressions are essentially same. 
\end{remark}

We have another expression of the solutions. 
Regard the time dependent mass as an inhomogeneous term, namely, 
\begin{align*}
	\partial_t^2 v_l - \Delta v_l =-\frac{\mu}{(1+t)^2} v_l.
\end{align*}
Then, we have the following expression.
\begin{align}
\label{eq2.3}
	v_l(t) = \mathcal{W}_0(t-t_0) v_l(t_0) + \mathcal{W}_1(t-t_0) \partial_t v_l (t_0) + \int_{t_0}^{t} \mathcal{W}_1(t-s) \frac{-\mu}{(1+s)^2} v_l(s)ds,
\end{align}
where $\mathcal{W}_0:= \cos (t|\nabla|)$ and $\mathcal{W}_1:=|\nabla|^{-1}\sin (t|\nabla|)$ are the propagators of the free wave equation. 
By the vector expression, we have
\begin{align*}
	\begin{pmatrix}
	v_l \\ \dot{v_l}
	\end{pmatrix}
	= \mathcal{W}(t-t_0)
	\begin{pmatrix}
	v_{l,0} \\ v_{l,1}
	\end{pmatrix}
	+ \int_{t_0}^{t}  \mathcal{W}(t-s) F(s,u(s)) ds
\end{align*}
where 
\begin{align*}
	\mathcal{W}(t):=
	\begin{pmatrix}
	 \mathcal{W}_{0}(t) &  \mathcal{W}_{1}(t)
	 \\
	 \dot{\mathcal{W}_{0}}(t) &  \dot{\mathcal{W}_{1}}(t)
	\end{pmatrix},
	\quad
	F(t,u) :=
	\begin{pmatrix}
	0
	\\
	-\frac{\mu}{(1+t)^2} u
	\end{pmatrix}.
\end{align*}


\subsection{The Strichartz estimates}

We recall the energy estimates for the solution to \eqref{eq2.1} by B\"{o}hme and Reissig \cite{BoRe12}.
\begin{lemma}[{\cite[Theorem 4]{BoRe12}}]
\label{lem2.1}
Let $\mu > 0$. Then, the solution $v_l$ of \eqref{eq2.1} satisfies the following estimates.
\begin{align*}
	&\| v_l(t) \|_{\dot{H}^1}  \cleq \| v_{l,0} \|_{H^1} +\| v_{l,1} \|_{L^2},
	\\
	&\| \partial_t v_l(t) \|_{L^2} \cleq \| v_{l,0} \|_{H^1} +\| v_{l,1} \|_{L^2},
\end{align*}
where the implicit constants do not depend on $t_0$ and $t$. 
Moreover, we have the following estimate.
\begin{align*}
	&\| v_l(t) \|_{L^2} \cleq
	\begin{cases}
	(1+t)^{\frac{1}{2}+\re\nu} (1+t_0)^{\frac{1}{2} - \re\nu} (\| v_{l,0} \|_{L^2}+\| v_{l,1} \|_{L^2}) & \text{ when } \mu \neq \frac{1}{4},
	\\
	(1+t)^{\frac{1}{2}} (1+t_0)^{\frac{1}{2}} \left(1+\log\frac{1+t}{1+t_0}\right) (\| v_{l,0} \|_{L^2}+\| v_{l,1} \|_{L^2}) & \text{ when } \mu =\frac{1}{4},
	\end{cases} 
\end{align*}
where the implicit constants are independent of $t_0$ and $t$. 
\end{lemma}

\begin{remark}
Lemma \ref{lem2.1} will be used only when $\mu>0$. If $\mu=0$, then we do not use this since the mass term, i.e. the $L^2$-norm of the solution, does not appear. 
\end{remark}


We also recall the Strichartz estimates for the free wave equation. 
We say that $(q,r)$ is a wave admissible pair if $(q,r)$ satisifes $2\leq q,r \leq \infty$ and
\begin{align*}
	\frac{1}{q} \leq \frac{d-1}{2} \left( \frac{1}{2} - \frac{1}{r}\right)
	\text{ and }
	(q,r,d) \neq (2,\infty,3)
\end{align*}
We set $\gamma:= d(1/2-1/r)-1/q$ to denote the derivative loss of the wave propagator.

\begin{lemma}[\cite{GiVe95,KeTa98}]
\label{lem2.2}
Let $(q,r)$ be a wave admissible pair. Then, we have
\begin{align*}
	&\norm{\mathcal{W}_{0}(t) \phi}_{L^q(I:L^r)} \cleq \norm{\phi}_{\dot{H}^{\gamma}}, 
	\\
	&\norm{\mathcal{W}_{1}(t) \phi}_{L^q(I:L^r)} \cleq \norm{\phi}_{\dot{H}^{\gamma-1}}, 
	\\
	&\norm{\int_{t_0}^{t}\mathcal{W}_{1}(t-s) F(s)ds}_{L^q(I:L^r)} \cleq \norm{F}_{L_t^1(I:L^2)}, 
\end{align*}
where $I$ is an interval containing $t_0$ and the implicit constants are independent of $I$. 
\end{lemma}

Combining these estimates, we get the Strichartz estimates for the propagator $\mathcal{E}$. 

\begin{proposition}
\label{prop2.3}
Let $\mu \geq 0$ and $(q,r)$ be a wave admissible pair such that $\gamma=1$. We have the following estimates.
\begin{align*}
	&\norm{\mathcal{E}_{0}(t,t_0) \phi}_{L_t^q(I: L_x^r)}  \cleq \norm{\phi}_{\dot{H}^1} + \mu \norm{\phi}_{L^2},
	\\
	&\norm{\mathcal{E}_{1}(t,t_0) \phi}_{L_t^q (I:L_x^r)}  \cleq (1+\mu)\norm{\phi}_{L^2},
\end{align*}
where $I=[t_0,T)$ for any $T >t_0$ and the implicit constants do not depend on $T$ and $t_0$. 
\end{proposition}

\begin{proof}
When $\mu=0$, $\mathcal{E}_{0} = \mathcal{W}_{0}$ and $\mathcal{E}_{1}=\mathcal{W}_{1}$ and thus the statement holds by Lemma \ref{lem2.2}. We consider $\mu>0$. 
Let $v_l$ be a solution of the linear equation \eqref{eq2.1}. 
By \eqref{eq2.3} and the Strichartz estimates for the wave equation, Lemma \ref{lem2.2}, it holds that
\begin{align*}
	\norm{v_l}_{L_t^q L_x^r(I)} \cleq \norm{v_{l,0}}_{\dot{H}^1}+ \norm{v_{l,1}}_{L^2} +\mu \norm{(1+t)^{-2}v_l}_{L_t^1L_x^2(I)},
\end{align*}
where the implicit constant does not depend on $T$ and $t_0$. 
When $\mu\neq1/4$, by Lemma \ref{lem2.1},  we have
\begin{align*}
	 \norm{(1+t)^{-2}v_l}_{L_t^1 L_x^2(I)}
	&\cleq \norm{(1+t)^{-\frac{3}{2} +\re\nu} }_{L_t^1(I)} (1+t_0)^{\frac{1}{2} - \re\nu} (\| v_{l,0} \|_{L^2}+\| v_{l,1} \|_{L^2}) 
	\\
	&\cleq (1+t_0)^{-\frac{1}{2} +\re\nu}(1+t_0)^{\frac{1}{2} - \re\nu}(\| v_{l,0} \|_{L^2}+\| v_{l,1} \|_{L^2})
	\\
	&\cleq \| v_{l,0} \|_{L^2}+\| v_{l,1} \|_{L^2}
\end{align*}
On the other hand, if $\mu=1/4$, then it holds from Lemma \ref{lem2.1} that
\begin{align*}
	 \norm{(1+t)^{-2}v_l}_{L_t^1 L_x^2(I)}
	&\cleq \norm{(1+t)^{-\frac{3}{2}} \left(1+\log \frac{1+t}{1+t_0}\right) }_{L_t^1(I)} (1+t_0)^{\frac{1}{2}} (\| v_{l,0} \|_{L^2}+\| v_{l,1} \|_{L^2}) 
	\\
	&\cleq (1+t_0)^{-\frac{1}{2}}(1+t_0)^{\frac{1}{2}}(\| v_{l,0} \|_{L^2}+\| v_{l,1} \|_{L^2})
	\\
	&\cleq \| v_{l,0} \|_{L^2}+\| v_{l,1} \|_{L^2}
\end{align*}

Therefore, in any case, we obtain
\begin{align*}
	\norm{v_l}_{L_t^q L_x^r(I)} 
	&\cleq  \norm{v_{l,0}}_{\dot{H}^1} + \norm{v_{l,1}}_{L^2} +\mu(\| v_{l,0} \|_{L^2}+\| v_{l,1} \|_{L^2})
\end{align*}
for a wave admissible pair $(q,r)$ satisfying $\gamma=1$. This implies the statements. 
\end{proof}
\section{Scattering for the linear equation}

We give the proofs of Theorem \ref{thm1.0} and Corollary \ref{cor1.0}. 
First, we show the existence of the limit of $\overrightarrow{v_l}(t)$ in $\dot{H}^1(\mathbb{R}^d) \times L^2(\mathbb{R}^d)$.

\begin{lemma}
Let $\mu \geq 0$. Then, there exists $\overrightarrow{v_{+}} \in \dot{H}^1(\mathbb{R}^d) \times L^2(\mathbb{R}^d)$ such that
\begin{align*}
	\norm{\overrightarrow{v_l}(t) - \mathcal{W}(t)\overrightarrow{v_{+}}}_{\dot{H}^1 \times L^2}
	\to 0
\end{align*}
as $t \to \infty$.  
\end{lemma}

\begin{proof}
If $\mu=0$, the statement is trivial. We only consider the case $\mu\neq 1/4$ and $\mu>0$. In the case of $\mu=1/4$, the statement also holds by a small modification.
For $0\leq \tau \leq t$, by Lemma \ref{lem2.1}, we get
\begin{align*}
	&\norm{\mathcal{W}(-t)\overrightarrow{v_l}(t)- \mathcal{W}(-\tau)\overrightarrow{v_l}(\tau)}_{\dot{H}^1\times L^2}
	\\
	&\cleq \norm{ \int_{\tau}^{t} \mathcal{W}_1(-s) \frac{\mu}{(1+s)^2} v_l(s)ds}_{\dot{H}^1}
	+\norm{ \int_{\tau}^{t} \mathcal{W}_0(-s) \frac{\mu}{(1+s)^2} v_l(s)ds}_{L^2}
	\\
	&\cleq \mu\int_{\tau}^{t} (1+s)^{-2} \norm{v_l(s)}_{L^2} ds
	\\
	&\cleq  \mu\int_{\tau}^{t} (1+s)^{-2+\frac{1}{2}+\re \nu}  ds (\| v_{l,0} \|_{L^2}+\| v_{l,1} \|_{L^2})
	\\
	&\to 0
\end{align*}
as $\tau \to \infty$ since $\re \nu < 1/2$ when $\mu >0$.
By the completeness of $\dot{H}^1\times L^2$, we get the limit $\overrightarrow{v_+}$.
\end{proof}

Next, we show the asymptotic order. 

\begin{proof}[Proof of Theorem \ref{thm1.0}]
When $\mu=0$, this is trivial. 
First, we consider the case of $\mu \neq 1/4$. It holds from Lemma \ref{lem2.1} that
\begin{align*}
	&\norm{\overrightarrow{v_l}(t)- \mathcal{W}(t)\overrightarrow{v_+}}_{\dot{H}^1\times L^2}
	\\
	&\cleq \mu\int_{t}^{\infty} (1+s)^{-2} \norm{v_l(s)}_{L^2} ds
	\\
	&\cleq  \mu  (1+t)^{-\frac{1}{2}+\re \nu}  (\| v_{l,0} \|_{L^2}+\| v_{l,1} \|_{L^2}).
\end{align*}
Next, in the case of $\mu = 1/4$, by the similar argument, we have
\begin{align*}
	&\norm{\overrightarrow{v_l}(t)- \mathcal{W}(t)\overrightarrow{v_+}}_{\dot{H}^1\times L^2}
	\\
	&\cleq \mu\int_{t}^{\infty} (1+s)^{-2} \norm{v_l(s)}_{L^2} ds
	\\
	&\cleq \mu\int_{t}^{\infty} (1+s)^{-2+\frac{1}{2}} (1+\log (1+s))  ds (\| v_{l,0} \|_{L^2}+\| v_{l,1} \|_{L^2})
	\\
	&\cleq \mu (1+t)^{-\frac{1}{2}} (1+\log(1+t)) (\| v_{l,0} \|_{L^2}+\| v_{l,1} \|_{L^2}).
\end{align*}
This completes the proof. 
\end{proof}

At last, we go back to \eqref{DW}. 

\begin{proof}[Proof of Corollary \ref{cor1.0}]
By retransforming $u_l=(1+t)^{-\mu_1/2}v_l$, it is easy to show that 
\begin{align*}
	&\norm{u_l(t) -(1+t)^{-\frac{\mu_1}{2}} (\mathcal{W}(t)\overrightarrow{v_{+}} )_1}_{\dot{H}^1}
	\\
	&\cleq \mu 
	(\norm{u_{l,0}}_{L^2}+\norm{u_{l,1}}_{L^2})
	\begin{cases}
	(1+t)^{-\frac{1}{2} + \re\nu - \frac{\mu_1}{2}}, & \mu \neq 1/4,
	\\
	(1+t)^{-\frac{1}{2} - \frac{\mu_1}{2}} \log(1+t), & \mu =1/4.
	\end{cases}
\end{align*}
We treat the time derivative of $u_l$. Since $\partial_t v_l = \frac{\mu_1}{2}(1+t)^{\frac{\mu_1}{2}-1}u_l + (1+t)^{\frac{\mu_1}{2}} \partial_t u_l$, we have
\begin{align*}
	&\norm{(1+t)^{\frac{\mu_1}{2}} \partial_t u_l(t) -( \mathcal{W}(t)\overrightarrow{v}_{+})_{2} }_{L^2} 
	\\
	&\cleq \norm{ \partial_t v_l(t) -( \mathcal{W}(t)\overrightarrow{v}_{+})_{2} }_{L^2}+ \frac{|\mu_1|}{2} (1+t)^{\frac{\mu_1}{2}-1} \norm{u_l(t)}_{L^2}.
\end{align*}
The first term is estimated by Theorem \ref{thm1.0}. For the second term, by Lemma \ref{lem2.1} and $u_l=(1+t)^{-\mu_1/2}v_l$, we get
\begin{align*}
	(1+t)^{\frac{\mu_1}{2}-1} \norm{u_l(t)}_{L^2}
	&=(1+t)^{-1} \norm{ v_l(t)}_{L^2}
	\\
	&\cleq (\norm{u_{l,0}}_{L^2}+ \norm{u_{l,1}}_{L^2}) 
	\begin{cases}
	 (1+t)^{-\frac{1}{2}+\re \nu}, & \mu\neq 1/4,
	 \\
	 (1+t)^{-\frac{1}{2}} \log(1+t), & \mu=1/4.
	\end{cases}
\end{align*}
Therefore, we get
\begin{align*}
	&\norm{(1+t)^{\frac{\mu_1}{2}} \partial_t u_l(t) -( \mathcal{W}(t)\overrightarrow{v}_{+})_{2} }_{L^2} 
	\\
	&\cleq (\mu+|\mu_1|)(\norm{u_{l,0}}_{L^2}+ \norm{u_{l,1}}_{L^2}) 
	\begin{cases}
	 (1+t)^{-\frac{1}{2}+\re \nu}, & \mu\neq 1/4,
	 \\
	 (1+t)^{-\frac{1}{2}} \log(1+t), & \mu=1/4.
	\end{cases}
\end{align*}
This implies the desired statement.
\end{proof}
\section{Small data scattering for the nonlinear equation}

We give the proofs of Theorem \ref{thm1.1}, Theorem \ref{thm1.3}, and Corollary \ref{cor1.6}.

First, we show the local well-posedness and the small data global existence for \eqref{NLKG}.

\begin{proof}[Proof of Theorem \ref{thm1.1}]
Now, $(p_1,2p_1)$ is a wave admissible pair satisfying $\gamma=1$ since $3\leq d \leq 6$. 
We set 
\begin{align*}
	\Phi[v](t):=\mathcal{E}_{0}(t,0) v_0 + \mathcal{E}_{1}(t,0) v_1 + \int_{0}^{t} \mathcal{E}_{1}(t,s) \frac{\lambda}{(1+s)^{\frac{2\mu_1}{d-2}}} |v(s)|^{\frac{4}{d-2}} v(s) ds,
\end{align*}
and
\begin{align*}
	X(a,T)&:=\{ v : [0,T) \times \mathbb{R}^d\to \mathbb{R} : \| v \|_{X} \leq a\},
	\\
	\| v \|_{X}&:=\| v \|_{L_t^{p_1}L_x^{2p_1}(0,T)}.
\end{align*}
We show that $\Phi$ is a contraction mapping on $X(a,T)$ for some $a>0$. 

By the Strichartz estimates for $\mathcal{E}$, i.e. Proposition \ref{prop2.3}, and $\mu_1\geq 0$, we get
\begin{align*}
	\norm{\Phi[v]}_{X}  
	&\cleq \delta 
	+  \int_{0}^{T} \norm{ \chi_{[0,t]}(s) \mathcal{E}_{1}(t,s) \frac{\lambda}{(1+s)^{\frac{2\mu_1}{d-2}}} |v(s)|^{\frac{4}{d-2}} v(s)}_{X} ds
	\\
	&\cleq \delta
	+  \int_{0}^{T} \frac{1}{(1+s)^{\frac{2\mu_1}{d-2}}} \norm{ |v(s)|^{\frac{4}{d-2}} v(s)}_{L_x^2} ds
	\\
	&\cleq \delta
	+\norm{ |v|^{\frac{4}{d-2}} v}_{L_t^1L_x^2(0,T)}
	\\
	&\cleq \delta
	+\norm{v}_{X}^{p_1}
	\\
	&\leq C\delta
	+Ca^{p_1}.
\end{align*}
If we take $\delta$ and $a$ satisfying $Ca^{p_1-1}<1/2$ and $C\delta < a/2$, then $\| \Phi[v] \|_{X} \leq a$ and thus $\Phi$ is a mapping on $X(a,T)$. Take $v$ and $\tilde{v}$ from $X(a,T)$. Then, 
\begin{align*}
	\norm{\Phi[v] - \Phi[\tilde{v}]}_{X}  
	&\cleq \int_{0}^{T} \norm{ \chi_{[0,t]}(s) \mathcal{E}_{1}(t,s) \frac{\lambda}{(1+s)^{\frac{2\mu_1}{d-2}}} (|v(s)|^{\frac{4}{d-2}}v(s)-|\tilde{v}(s)|^{\frac{4}{d-2}} \tilde{v}(s))} _{X} ds
	\\
	&\cleq  \int_{0}^{T} \frac{1}{(1+s)^{\frac{2\mu_1}{d-2}}} \norm{ (|v(s)|^{\frac{4}{d-2}}v(s)-|\tilde{v}(s)|^{\frac{4}{d-2}} \tilde{v}(s))}_{L_x^2} ds
	\\
	&\cleq \norm{ |v(s)|^{\frac{4}{d-2}}v(s)-|\tilde{v}(s)|^{\frac{4}{d-2}} \tilde{v}(s)}_{L_t^1L_x^2(0,T)}
	\\
	&\cleq (\norm{v}_{X}^{p_1-1} + \norm{\tilde{v}}_{X}^{p_1-1})\norm{v-\tilde{v}}_{X}
	\\
	&\leq Ca^{p_1-1}\norm{v-\tilde{v}}_{X}.
\end{align*}
If $Ca^{p_1-1}<1/2$, then $\Phi$ is a contraction mapping on $X(a,T)$ and thus there exists a unique $v \in X(a,T)$ such that $v=\Phi[v]$. Moreover, $v(t)$ belongs to $H^1$ for each $t\in [0,T)$. Indeed, by Lemma \ref{lem2.1} and $\mu_1\geq 0$, we have
\begin{align*}
	\norm{v(t)}_{\dot{H}^1} 
	&\cleq \norm{v_0}_{H^1} + \norm{v_1}_{L^2} 
	+  \int_{0}^{T} \norm{\chi_{[0,t]}(s) \mathcal{E}_{1}(t,s) \frac{\lambda}{(1+s)^{\frac{2\mu_1}{d-2}}} |v(s)|^{\frac{4}{d-2}} v(s)}_{\dot{H}^1} ds
	\\
	&\cleq \norm{v_0}_{H^1} + \norm{v_1}_{L^2} 
	+  \int_{0}^{T} \frac{\lambda}{(1+s)^{\frac{2\mu_1}{d-2}}} \norm{ |v(s)|^{\frac{4}{d-2}} v(s)}_{L^2} ds
	\\
	&\cleq \norm{v_0}_{H^1} + \norm{v_1}_{L^2} 
	+\norm{v}_{X}^{p_1}.
\end{align*}
For the $L^2$-norm, we have the following. If $\mu \neq 1/4$, then it holds that
\begin{align*}
	\norm{v(t)}_{L^2} 
	&\cleq (1+t)^{\frac{1}{2}+\re\nu}(\norm{v_0}_{H^1} + \norm{v_1}_{L^2} )
	\\
	&\quad +  \int_{0}^{T} \norm{\chi_{[0,t]}(s) \mathcal{E}_{1}(t,s) \frac{\lambda}{(1+s)^{\frac{2\mu_1}{d-2}}} |v(s)|^{\frac{4}{d-2}} v(s)}_{L^2} ds
	\\
	&\cleq (1+t)^{\frac{1}{2}+\re\nu}(\norm{v_0}_{H^1} + \norm{v_1}_{L^2} )
	\\
	&\quad +  (1+t)^{\frac{1}{2}+\re\nu} \int_{0}^{T}\chi_{[0,t]}(s) \frac{(1+s)^{\frac{1}{2}-\re\nu}}{(1+s)^{\frac{2\mu_1}{d-2}}}  \norm{ |v(s)|^{\frac{4}{d-2}} v(s)}_{L^2} ds
	\\
	&\cleq  (1+t)^{\frac{1}{2}+\re\nu}(\norm{v_0}_{H^1} + \norm{v_1}_{L^2} )
	+  (1+t)^{\max\{\frac{1}{2}+\re\nu, 1-\frac{2\mu_1}{d-2}\}} \norm{v}_{X}^{p_1}
	\\
	&\cleq   (1+t)^{\max\{\frac{1}{2}+\re\nu, 1-\frac{2\mu_1}{d-2}\}}(\norm{v_0}_{H^1} + \norm{v_1}_{L^2} +  \norm{v}_{X}^{p_1}).
\end{align*}
In the similar way, if $\mu=1/4$, then we have
\begin{align*}
		\norm{v(t)}_{L^2} 
	&\cleq (1+t)^{\frac{1}{2}}(1+\log (1+t))(\norm{v_0}_{H^1} + \norm{v_1}_{L^2} )
	\\
	&\quad +  (1+t)^{\frac{1}{2}} \int_{0}^{T}\chi_{[0,t]}(s) \frac{(1+s)^{\frac{1}{2}}}{(1+s)^{\frac{2\mu_1}{d-2}}}  \left(1+\log \frac{1+t}{1+s}\right) \norm{ |v(s)|^{\frac{4}{d-2}} v(s)}_{L^2} ds
	\\
	&\cleq  (1+t)^{\frac{1}{2}}(1+\log (1+t))(\norm{v_0}_{H^1} + \norm{v_1}_{L^2} )
	\\
	&\quad +  (1+t)^{\max\{\frac{1}{2}, 1-\frac{2\mu_1}{d-2}\}} (1+\log (1+t))\norm{v}_{X}^{p_1}
	\\
	&\cleq   (1+t)^{\max\{\frac{1}{2}, 1-\frac{2\mu_1}{d-2}\}}(1+\log (1+t))(\norm{v_0}_{H^1} + \norm{v_1}_{L^2} +  \norm{v}_{X}^{p_1}).
\end{align*}
In any case, $v(t) \in H^1$ and we also obtain $\partial_t v(t)  \in L^2$ in the same way, . 
\end{proof}

\begin{remark}
When $\mu_1<0$, the local well-posedness also holds and, however, the global existence is not clear.  
\end{remark}

\begin{proof}[Proof of Corollary \ref{cor1.4}]
By the Strichartz estimates, Proposition \ref{prop2.3}, we have
\begin{align*}
	\| \mathcal{E}_{0}(t,0) v_0 + \mathcal{E}_{1}(t,0) v_1\|_{L^{p_1}([0,\infty): L_x^{2p_1})} \leq C(\norm{v_0}_{H^1} + \norm{v_1}_{L^2})
\end{align*}
and thus, if $\norm{v_0}_{H^1} + \norm{v_1}_{L^2} < \delta/C$, then we can take $T=\infty$ in Theorem \ref{thm1.1}. Therefore, we get the small data global existence. 
\end{proof}

\begin{corollary}
\label{cor3.1}
Let $v$ be a global solution to \eqref{NLKG} satisfying $\norm{v}_{L_t^{p_1}L_x^{2p_1}(0,\infty)}<\infty$.
Then it is valid that
\begin{align*}
	\| v(t) \|_{L^2} \cleq 
	\begin{cases}
	(1+t)^{\max\{\frac{1}{2}+\re\nu, 1-\frac{2\mu_1}{d-2}\}} &\text{ if } \mu\neq 1/4,
	\\
	(1+t)^{\max\{\frac{1}{2}, 1-\frac{2\mu_1}{d-2}\}}(1+\log (1+t))&\text{ if } \mu= 1/4,
	\end{cases}
\end{align*}
for all $t\in [0,\infty)$, where the implicit constant depends on $\norm{v_0}_{H^1}$, $\norm{v_1}_{L^2}$, and $\norm{v}_{L_t^{p_1}L_x^{2p_1}}$ and is independent of time. 
\end{corollary}

\begin{proof}
As the last in the proof of Theorem \ref{thm1.1}, we get the desired estimate. 
\end{proof}

Next, we prove the scattering and its asymptotic order for \eqref{NLKG}. 

\begin{proof}[Proof of Theorem \ref{thm1.3}]
If $v$ is the global solution to \eqref{NLKG} satisfying $\|v\|_{L_t^{p_1}L_x^{2p_1}(0,\infty)}< \infty$, then $v$ also satisfies 
the following integral equation, where we use $d \leq 5$ (see Appendix \ref{appB} for the details). 
\begin{align*}
	v(t) &= \mathcal{W}_0(t) v_0 + \mathcal{W}_1(t) v_1
	\\
	& + \int_{0}^{t} \mathcal{W}_1(t-s) \left( \frac{-\mu}{(1+s)^2}v(s)+  \frac{\lambda}{(1+s)^{\frac{2\mu_1}{d-2}}} |v(s)|^{\frac{4}{d-2}} v(s) \right)ds.
\end{align*}
Let $\overrightarrow{v}=(v,\partial_t v)$. Then, we have
\begin{align*}
	\mathcal{W}(-t)\overrightarrow{v}(t) &= \overrightarrow{v}(0)+ \int_{0}^{t} \mathcal{W}(-s)\overrightarrow{N}(s, v(s)) ds,
\end{align*}
where 
\begin{align*}
	\overrightarrow{N}(s,v):=
	\begin{pmatrix}
	0
	\\
	\frac{-\mu}{(1+s)^2} v(s)+  \frac{\lambda}{(1+s)^{\frac{2\mu_1}{d-2}}} |v(s)|^{\frac{4}{d-2}} v(s)
	\end{pmatrix}.
\end{align*}
For $0\leq \tau \leq t$, it holds from $\mu_1> 0$, which implies $1-2\mu_1/(d-2) <1$, that
\begin{align*}
	&\norm{\mathcal{W}(-t)\overrightarrow{v}(t)- \mathcal{W}(-\tau)\overrightarrow{v}(\tau)}_{\dot{H}^1\times L^2}
	\\
	&\cleq \norm{ \int_{\tau}^{t} \mathcal{W}_1(-s) \frac{\mu}{(1+s)^2} v(s)ds}_{\dot{H}^1}
	+\norm{ \int_{\tau}^{t} \mathcal{W}_0(-s) \frac{\mu}{(1+s)^2} v(s)ds}_{L^2}
	\\
	& \quad +\norm{ \int_{\tau}^{t} \mathcal{W}_1(-s) \frac{ |v(s)|^{\frac{4}{d-2}} v(s) }{(1+s)^{\frac{2\mu_1}{d-2}}}ds}_{\dot{H}^1}
	+\norm{ \int_{\tau}^{t} \mathcal{W}_0(-s) \frac{ |v(s)|^{\frac{4}{d-2}} v(s) }{(1+s)^{\frac{2\mu_1}{d-2}}}ds}_{L^2}
	\\
	&\cleq  \mu \int_{\tau}^{t} (1+s)^{-2} \norm{v(s)}_{L^2} ds
	+\int_{\tau}^{t} \frac{1}{(1+s)^{\frac{2\mu_1}{d-2}}}\norm{  |v(s)|^{\frac{4}{d-2}} v(s) }_{L^2}ds
	\\
	&\cleq  \mu \norm{ (1+s)^{-2} v}_{L_s^1L^2(\tau,t)} 
	+\norm{v}_{L^{p_1}L^{2p_1}(\tau,t)}^{p_1}
	\\
	&\to 0
\end{align*}
as $\tau,t\to \infty$ since $\mu \|(1+t)^{-2} v\|_{L_t^1L^2(0,\infty)}+\|v\|_{L^{p_1}L^{2p_1}(0,\infty)}^{p_1}<\infty$ by the assumption and Corollary \ref{cor3.1}. 

Thus, $\mathcal{W}(-t)\overrightarrow{v}(t)$ converges to some $\overrightarrow{v_+}$ in $\dot{H}^1 \times L^2$. Here, $\overrightarrow{v_+}$ is given by 
\begin{align*}
	\overrightarrow{v_+}= \overrightarrow{v}(0)+ \int_{0}^{\infty} \mathcal{W}(-s)\overrightarrow{N}(s, v(s)) ds.
\end{align*}
We set 
$\alpha:=\max\{\frac{1}{2}+\re\nu, 1-\frac{2\mu_1}{d-2}\}$. We consider the case $\mu \neq1/4$. 
Since $\mathcal{W}$ is a unitary operator on $\dot{H}^1 \times L^2$, by Lemma \ref{lem2.2} and Corollary \ref{cor3.1}, we obtain
\begin{align*}
	&\norm{\overrightarrow{v}(t) - \mathcal{W}(t)\overrightarrow{v_{+}}}_{\dot{H}^1 \times L^2}
	\\
	&=\norm{ \int_{t}^{\infty} \mathcal{W}(t-s)\overrightarrow{N}(s, v(s)) ds}_{\dot{H}^1 \times L^2}
	\\
	&\cleq \mu \int_{t}^{\infty} (1+s)^{-2} \norm{v(s)}_{L^2} ds
	+\int_{t}^{\infty} \frac{1}{(1+s)^{\frac{2\mu_1}{d-2}}}\norm{  |v(s)|^{\frac{4}{d-2}} v(s) }_{L^2}ds
	\\
	&\cleq \mu \int_{t}^{\infty} (1+s)^{-2+\alpha} ds 
	+(1+t)^{-\frac{2\mu_1}{d-2}} \int_{t}^{\infty} \norm{v}_{L_x^{2p_1}}^{p_1}ds
	\\
	&\cleq \mu (1+t)^{-1+\alpha} 
	+(1+t)^{-\frac{2\mu_1}{d-2}} \norm{v}_{L^{p_1}L^{2p_1}(t,\infty)}^{p_1}
	\\
	&= \mu (1+t)^{\max\{-\frac{1}{2}+\re \nu, -\frac{2\mu_1}{d-2}\}} 
	+(1+t)^{-\frac{2\mu_1}{d-2}} o_t(1),
\end{align*}
where $o_t(1)=\norm{v}_{L^{p_1}L^{2p_1}(t,\infty)}^{p_1}$. 
In the case of $\mu=1/4$, by the same argument, we have
\begin{align*}
	&\norm{\overrightarrow{v}(t) - \mathcal{W}(t)\overrightarrow{v_{+}}}_{\dot{H}^1 \times L^2}
	\\
	&\cleq \mu (1+t)^{\max\{-\frac{1}{2},-\frac{2\mu_1}{d-2}\}}(1+\log(1+t)) 
	+(1+t)^{-\frac{2\mu_1}{d-2}}o_t(1).
\end{align*}
This completes the proof. 
\end{proof}

\begin{proof}[Proof of Corollary {\ref{cor1.6}}]
Combining the argument in Corollary \ref{cor1.0} and the result in Theorem \ref{thm1.3}, we get the statement. Thus, we omit the details. 
\end{proof}


%

\appendix

\section{Energy estimates for the linear solution}

For the readers' convenience, we give the proof of Lemma \ref{lem2.1} since the independence of the initial time $t_0$ plays an important role to obtain the Strichartz estimates for $\mathcal{E}$. 
By the Plancherel theorem and the H\"{o}lder estimate, it holds that
\begin{align*}
	\| v_l(t) \|_{\dot{H}^{s}} 
	&\leq  \| |\xi|^s E_{0}(t,t_0,\xi) \widehat{v_{l,0}}(\xi)\|_{L_{\xi}^2} 
	+\| |\xi|^s E_{1}(t,t_0,\xi)\widehat{v_{l,1}}(\xi)\|_{L_{\xi}^2},
	\\
	\| \partial_t v_l(t) \|_{L^2} 
	&\leq  \| \dot{E_{0}}(t,t_0,\xi) \widehat{v_{l,0}}(\xi)\|_{L_{\xi}^2} 
	+\| \dot{E_{1}}(t,t_0,\xi)\widehat{v_{l,1}}(\xi)\|_{L_{\xi}^2},
\end{align*}
for $s=0,1$. 
Therefore, it is enough to estimate $E_{j}(t,t_0,\xi)$ and $\dot{E_{j}}(t,t_0,\xi)$ for $j=0,1$

\subsection{Some lemmas for the Bessel functions}
To estimate $E_{j}$ and $\dot{E_{j}}$, we use well-known estimates for $J_{\nu}$ and $Y_{\nu}$ as follows. 
See \cite{Wat}. 

\begin{lemma}
\label{lemA.1}
Let $\nu \in \mathbb{C}$. Then, there exists $N>0$ such that the estimates 
\begin{align*}
	|J_{\nu}(\tau)| \leq C\tau^{-\frac{1}{2}},
	\quad \text{ and } \quad
	|Y_{\nu}(\tau)| \leq C\tau^{-\frac{1}{2}}
\end{align*}
hold for $\tau \geq N$. 
\end{lemma}

\begin{proof}
See \cite[7.21, (1) \& (2) in p.199]{Wat}.
\end{proof}

\begin{lemma}
\label{lemA.2}
Let $R>0$ be arbitrarily fixed and $\tau \leq R$. Then, we have
\begin{align*}
	|J_{\nu}(\tau)| \leq C \tau^{\re \nu}
\end{align*}
for $\nu \in \mathbb{C}$. 
\end{lemma}

\begin{proof}
This estimate comes from the definition of $J_{\nu}$.
\end{proof}

\begin{lemma}
\label{lemA.3}
We have
\begin{align*}
	|Y_{\nu}(\tau)| \leq C \tau^{-|\re\nu|}. 
\end{align*}
for $\nu \in \mathbb{C} \setminus \mathbb{Z}$. 
It is valid that 
\begin{align*}
	|Y_{n}(\tau)| \cleq \tau^{-n}
\end{align*}
for $n \in \mathbb{Z}_{\geq 1}$. 
Moreover, for $n \in \mathbb{Z}_{\geq 0}$, we also have
\begin{align*}
	Y_{-n}(\tau)&= \frac{2}{\pi} J_{-n}(\tau) \log \tau + f_{-n}(\tau)
\end{align*}
where $|f_{-n}|\cleq \tau^{-n}$ for $\tau \leq R$.
\end{lemma}

\begin{proof}
We apply the definition of $Y_{\nu}$ and the above lemma for the estimates when $\nu \in \mathbb{C} \setminus \mathbb{Z}$. 
For the statements when $n \in \mathbb{Z}$, see \cite[3.582, (4) in p.73]{Wat}.
\end{proof}

\subsection{Preliminaries}
\label{appA.2}
We recall $\nu:=\frac{1}{2}\sqrt{1-4\mu}$ when $\mu \leq 1/4$ and $\nu:=\frac{i}{2}\sqrt{4\mu-1}$ when $\mu > 1/4$ and $\tau = (1+t)|\xi|$. 

Now, we have
\begin{align*}
	\dot{e_{+}} = |\xi| \left(\frac{1}{2} \tau^{-\frac{1}{2}} J_{\nu}(\tau) + \tau^{\frac{1}{2}} \dot{J_{\nu}}(\tau)\right),
	\quad
	\dot{e_{-}} = |\xi| \left(\frac{1}{2} \tau^{-\frac{1}{2}} Y_{\nu}(\tau) + \tau^{\frac{1}{2}} \dot{Y_{\nu}}(\tau)\right).
\end{align*} 
Thus, we obtain
\begin{align*}
	e_{+} \dot{e_{-}} - \dot{e_{+}} e_{-}
	&=\tau^{\frac{1}{2}} J_{\nu}  |\xi| \left(\frac{1}{2} \tau^{-\frac{1}{2}} Y_{\nu} + \tau^{\frac{1}{2}} \dot{Y_{\nu}}\right)
	-\tau^{\frac{1}{2}} Y_{\nu}  |\xi| \left(\frac{1}{2} \tau^{-\frac{1}{2}} J_{\nu} + \tau^{\frac{1}{2}} \dot{J_{\nu}}\right)
	\\
	&=\tau  |\xi| \left(J_{\nu}\dot{Y_{\nu}} - Y_{\nu}\dot{J_{\nu}}\right)
	\\
	&=\tau |\xi| \left( \frac{2}{\pi \tau}\right)
	\\
	&= \frac{2}{\pi}|\xi|
\end{align*}
where we used the Wronskian of $J_{\nu}$ and $Y_{\nu}$ (see \cite[3.63 (1) in p.76]{Wat}).

\begin{lemma}
\label{lemA.4}
There exists $N >0$ such that 
\begin{align*}
	&|e_{\pm}(t,\xi)| \cleq 1,
	\\
	&|\dot{e_{\pm}}(t,\xi)| \cleq |\xi|.
\end{align*}
hold for $(1+t)|\xi| \geq N$,
\end{lemma}

\begin{proof}
We only estimate $e_{+}$. 
The other cases are treated similarly. 
Now, there exists $N>0$ such that $|J_{\nu}(\tau)| \leq C\tau^{-\frac{1}{2}}$ if $\tau \geq N$.
Therefore, the first inequality comes from
\begin{align*}
	|e_{+}(t,\xi)| =\tau^{\frac{1}{2}} |J_{\nu}(\tau)| \cleq \tau^{\frac{1}{2}} \tau^{-\frac{1}{2}} \ceq 1.
\end{align*}
We prove the second inequality. 
Since $\dot{J_{\nu}}= \frac{1}{2}(J_{\nu-1}- J_{\nu+1})$ for $\nu \in \mathbb{C}$,  
we obtain
\begin{align*}
	|\dot{e_{+}}(t,\xi)| 
	&\leq |\xi| \left(\frac{1}{2} \tau^{-\frac{1}{2}} |J_{\nu}| + \tau^{\frac{1}{2}} |\dot{J_{\nu}}|\right)
	\\
	&\cleq |\xi| \left\{ \tau^{-\frac{1}{2}} |J_{\nu}| + \tau^{\frac{1}{2}}(|J_{\nu-1}|+|J_{\nu+1}|)\right\}
	\\
	&\cleq |\xi| \left( \tau^{-\frac{1}{2}-\frac{1}{2}} + \tau^{\frac{1}{2} -\frac{1}{2}}+ \tau^{\frac{1}{2} -\frac{1}{2}}\right)
	\\
	&\cleq |\xi|.
\end{align*}
\end{proof}

\begin{lemma}
\label{lemA.5}
Let $R>0$. 
If $\mu \neq 1/4$, i.e., $\nu \neq 0$, then we have the following estimates.  
\begin{align*}
	&|e_{\pm}(t,\xi)| \cleq ((1+t)|\xi|)^{\frac{1}{2} \pm \re \nu},
	\\
	&|\dot{e_{\pm}}(t,\xi)| \cleq |\xi|((1+t)|\xi|)^{-\frac{1}{2} \pm \re \nu}
\end{align*}
for $(1+t)|\xi|\leq R$. 
If  $\mu = 1/4$, i.e., $\nu = 0$, then we have the following estimates.
\begin{align*}
	&|e_{+}(t,\xi)| \cleq ((1+t)|\xi|)^{\frac{1}{2}},  
	\\
	&|e_{-}(t,\xi)| \cleq ((1+t)|\xi|)^{\frac{1}{2}}(1+|\log((1+t)|\xi|)|),
	\\
	&|\dot{e_{+}}(t,\xi)| \cleq |\xi|((1+t)|\xi|)^{-\frac{1}{2} },
	\\
	&|\dot{e_{-}}(t,\xi)| \cleq |\xi|((1+t)|\xi|)^{-\frac{1}{2} }(1+|\log((1+t)|\xi|)|)
\end{align*}
 for $(1+t)|\xi|\leq R$. 
\end{lemma}

\begin{proof}
We only estimate $e_{+}$ in the case of $\mu \neq 1/4$, i.e., $\nu \neq 0$. 
The other cases are treated similarly. 
The first inequality is
\begin{align*}
	|e_{+}(t,\xi)| =\tau^{\frac{1}{2}} |J_{\nu}(\tau)| \cleq \tau^{\frac{1}{2}} \tau^{\re \nu}.
\end{align*}
Since $\dot{J_{\nu}}= \frac{1}{2}(J_{\nu-1}- J_{\nu+1})$ for $\nu \in \mathbb{C}$,  
we obtain the second inequality as follows. 
\begin{align*}
	|\dot{e_{+}}(t,\xi)| 
	&\leq |\xi| \left(\frac{1}{2} \tau^{-\frac{1}{2}} |J_{\nu}| + \tau^{\frac{1}{2}} |\dot{J_{ \nu}}|\right)
	\\
	&\cleq |\xi| \left\{ \tau^{-\frac{1}{2}} |J_{\nu}| + \tau^{\frac{1}{2}}(|J_{\nu-1}|+|J_{\nu+1}|)\right\}
	\\
	&\cleq |\xi| \left( \tau^{-\frac{1}{2}+\re \nu} + \tau^{\frac{1}{2}+\re \nu-1}+ \tau^{\frac{1}{2}+\re \nu+1}\right)
	\\
	&\cleq |\xi| \left( \tau^{-\frac{1}{2}+\re \nu} + \tau^{\frac{3}{2}+ \re \nu}\right)
	\\
	&\cleq |\xi| \tau^{-\frac{1}{2}+\re \nu}.
\end{align*}
\end{proof}

\begin{lemma}
\label{lemA.6}
Let $R>0$ and  $\tau:=(1+t)|\xi|$, $\tau_0:=(1+t_0)|\xi|$, and $\tau,\tau_0 \leq R$. 
If $\mu = 1/4$, i.e., $\nu = 0$, then we have the following estimates.  
\begin{align*}
	&|e_{+}(t,\xi) \dot{e_{-}}(t_0,\xi) - \dot{e_{+}}(t_0,\xi) e_{-}(t,\xi)|
	\cleq \tau^{\frac{1}{2}} \tau_{0}^{-\frac{1}{2}} \left(1+ \left| \log \frac{\tau}{\tau_0} \right| \right),
	\\
	&|e_{+}(t_0,\xi) e_{-}(t,\xi) - e_{+}(t,\xi) e_{-}(t_0,\xi)|
	\cleq \tau^{\frac{1}{2}} \tau_{0}^{\frac{1}{2}} \left(1+ \left| \log \frac{\tau}{\tau_0} \right| \right),
	\\
	&|\dot{e_{+}}(t,\xi) \dot{e_{-}}(t_0,\xi) - \dot{e_{+}}(t_0,\xi) \dot{e_{-}}(t,\xi)|
	\cleq \tau^{-\frac{1}{2}} \tau_{0}^{-\frac{1}{2}} \left(1+ \left| \log \frac{\tau}{\tau_0} \right| \right).
\end{align*}
\end{lemma}

\begin{proof}
We only prove the first inequality. The others can be shown in the similar way. We have
\begin{align*}
	&|e_{+}(t,\xi) \dot{e_{-}}(t_0,\xi) - \dot{e_{+}}(t_0,\xi) e_{-}(t,\xi)|
	\\
	&=
	\left| \tau^{\frac{1}{2}} J_{0}(\tau) \left(\frac{1}{2}\tau_0^{-\frac{1}{2}}Y_{0}(\tau_0)+\tau_0^{\frac{1}{2}}\dot{Y_0}(\tau_0)\right)
	-\left(\frac{1}{2}\tau_0^{-\frac{1}{2}}J_{0}(\tau_0)+\tau_0^{\frac{1}{2}}\dot{J_0}(\tau_0)\right) \tau^{\frac{1}{2}} Y_{0}(\tau)\right|
	\\
	&\cleq  \tau^{\frac{1}{2}} \tau_{0}^{-\frac{1}{2}}
	\left| J_{0}(\tau) Y_{0}(\tau_0)
	-J_{0}(\tau_0) Y_{0}(\tau)\right|
	+ \tau^{\frac{1}{2}} \tau_{0}^{\frac{1}{2}}
	\left| J_{0}(\tau) \dot{Y_{0}}(\tau_0)
	-\dot{J_{0}}(\tau_0) Y_{0}(\tau)\right|.
\end{align*}
For the first term, by Lemma \ref{lemA.3}, we obtain
\begin{align*}
	\left| J_{0}(\tau) Y_{0}(\tau_0)-J_{0}(\tau_0) Y_{0}(\tau)\right|
	&\cleq \left| J_{0}(\tau) J_{0}(\tau_0) (\log\tau_0  -  \log\tau ) \right| +1
	\\
	&\cleq  \left| \log \frac{\tau}{\tau_0} \right| +1.
\end{align*}
Since we have $\dot{J_0}=J_{-1}$ and $\dot{Y_0}=Y_{-1}$, for the second term, it holds from  Lemma \ref{lemA.3} that
\begin{align*}
	\left| J_{0}(\tau) \dot{Y_{0}}(\tau_0)-\dot{J_{0}}(\tau_0) Y_{0}(\tau)\right|
	&\cleq \left| J_{0}(\tau) J_{-1}(\tau_0) \log\tau_0-J_{-1}(\tau_0) Y_{0}(\tau)\right| +\tau_0^{-1}
	\\
	&\cleq | J_{0}(\tau) J_{-1}(\tau_0)| \left| \log \frac{\tau}{\tau_0} \right|+\tau_0^{-1}
	\\
	&\cleq \tau_0^{-1}  \left( \left| \log \frac{\tau}{\tau_0} \right|+1\right),
\end{align*}
where we also use Lemma \ref{lemA.2}. Combining these estimates, we get
\begin{align*}
	|e_{+}(t,\xi) \dot{e_{-}}(t_0,\xi) - \dot{e_{+}}(t_0,\xi) e_{-}(t,\xi)|
	\cleq \tau^{\frac{1}{2}} \tau_{0}^{-\frac{1}{2}} \left(1+ \left| \log \frac{\tau}{\tau_0} \right| \right).
\end{align*}
\end{proof}


\subsection{Estimates in the case of $\mu \neq 1/4$}
\label{appA.3}

We first consider $\mu \neq 1/4$, i.e., $\nu \neq 0$. 

Take large positive number $N$ such that the above estimate in Lemma \ref{lemA.4} holds if $\tau \geq N$.  
We set $\tau_0=(1+t_0)|\xi|$. 

\noindent{\bf Case 1-1.} $\tau \geq \tau_0 \geq N$.

By Lemma \ref{lemA.4} and $(e_{+} \dot{e_{-}} - \dot{e_{+}} e_{-})(t_0) \ceq |\xi|$, we get
\begin{align*}
	|E_{0}(t,t_0,\xi)| 
	\cleq |\xi|^{-1} (|e_{+}(t,\xi)||\dot{e_{-}}(t_0,\xi)| + |\dot{e_{+}}(t_0,\xi)||e_{-}(t,\xi)|) 
	\cleq 1
\end{align*}
and 
\begin{align*}
	|E_{1}(t,t_0,\xi)| 
	&\cleq |\xi|^{-1} (|e_{+}(t_0,\xi)||e_{-}(t,\xi)| + |e_{+}(t,\xi)||e_{-}(t_0,\xi)|) 
	\\
	&\cleq |\xi|^{-1}
	\cleq  (1+t)^{\frac{1}{2} + \re \nu } (1+t_0)^{\frac{1}{2} - \re \nu } .
\end{align*}
Moreover, we also have
\begin{align*}
	|\dot{E_{0}}(t,t_0,\xi)| 
	\cleq |\xi|
	\text{ and }
	|\dot{E_{1}}(t,t_0,\xi)| 
	\cleq 1.
\end{align*}

%
\noindent{\bf Case 1-2.} $\tau \geq N \geq \tau_0$. 
%
%

In this case, we have $|\xi| \leq N$ and
\begin{align*}
	|E_{0}(t,t_0,\xi)| 
	&\cleq |\xi|^{-1} (|e_{+}(t,\xi)||\dot{e_{-}}(t_0,\xi)| + |\dot{e_{+}}(t_0,\xi)||e_{-}(t,\xi)|) 
	\\
	&\cleq   \tau_0^{-\frac{1}{2} - \re \nu}
	+  \tau_0^{-\frac{1}{2} + \re \nu}
	\\
	&\cleq   \tau_0^{-\frac{1}{2} - \re \nu}
	\\
	&\ceq   (1+t_0)^{-\frac{1}{2} -\re \nu} |\xi|^{-\frac{1}{2} - \re \nu}
	\\
	&\cleq 
	\begin{cases}
	|\xi|^{-\frac{1}{2} -\re \nu}
	\\
	(1+t_0)^{-\frac{1}{2} - \re \nu} (1+t)^{\frac{1}{2} + \re \nu} 
	\end{cases},
\end{align*}
where we use $(1+t_0)^{-1}\leq 1$, $|\xi|^{-1} \cleq 1+t$ since $\tau \geq N$, and $1+t \geq 1+t_0$. 
We also estimate $E_1$ as follows.
\begin{align*}
	|E_{1}(t,t_0,\xi)| 
	&\cleq |\xi|^{-1} (|e_{+}(t_0,\xi)||e_{-}(t,\xi)| + |e_{+}(t,\xi)||e_{-}(t_0,\xi)|) 
	\\
	&\cleq |\xi|^{-1} (|e_{+}(t_0,\xi)| + |e_{-}(t_0,\xi)|)
	\\
	&\cleq |\xi|^{-1} (  \tau_0^{\frac{1}{2} + \re \nu}  +  \tau_0^{\frac{1}{2} - \re \nu}  )
	\\
	&\cleq |\xi|^{-1} \tau_0^{\frac{1}{2} - \re \nu}
	\\
	&\ceq   
	 (1+t_0)^{\frac{1}{2} - \re \nu}  |\xi|^{-\frac{1}{2}- \re \nu}
	\\
	&\cleq  
	(1+t_0)^{\frac{1}{2} - \re \nu} (1+t)^{\frac{1}{2}+ \re \nu}.
\end{align*}
We also have 
\begin{align*}
	|\dot{E_{0}}(t,t_0,\xi)| 
	&\cleq |\xi|^{-1} (|\dot{e_{+}}(t,\xi)||\dot{e_{-}}(t_0,\xi)| + |\dot{e_{+}}(t_0,\xi)||\dot{e_{-}}(t,\xi)|) 
	\\
	&\cleq |\xi|(|\dot{e_{-}}(t_0,\xi)| + |\dot{e_{+}}(t_0,\xi)|)
	\\
	&\cleq |\xi|\tau_0^{-\frac{1}{2} - \re \nu}
	+ |\xi| \tau_0^{-\frac{1}{2} + \re \nu} 
	\\
	&\cleq |\xi|\tau_0^{-\frac{1}{2} - \re \nu}
	\\
	&\ceq (1+t_0)^{-\frac{1}{2} - \re \nu} |\xi|^{\frac{1}{2} - \re \nu}
	\\
	&\cleq (1+t_0)^{-\frac{1}{2} - \re \nu} (1+t)^{-\frac{1}{2} + \re \nu}.
\end{align*}
Moreover, we have
\begin{align*}
	|\dot{E_{1}}(t,t_0,\xi)| 
	&\cleq |\xi|^{-1} (|e_{+}(t_0,\xi)||\dot{e_{-}}(t,\xi)| + |\dot{e_{+}}(t,\xi)||e_{-}(t_0,\xi)|) 
	\\
	&\cleq  \tau_0^{\frac{1}{2}+\re \nu} + \tau_0^{\frac{1}{2}-\re \nu}
	\\
	&\cleq \tau_0^{\frac{1}{2}-\re \nu}
	\\
	&\ceq 
	 (1+t_0)^{\frac{1}{2}-\re \nu}|\xi|^{\frac{1}{2}-\re \nu}
	\\
	&\cleq 
	(1+t_0)^{\frac{1}{2}-\re \nu}(1+t)^{-\frac{1}{2}+\re \nu}.
\end{align*}
\noindent{\bf Case 1-3.} $N \geq \tau \geq \tau_0$. 

We have
\begin{align*}
	|E_{0}(t,t_0,\xi)| 
	&\cleq |\xi|^{-1} (|e_{+}(t,\xi)||\dot{e_{-}}(t_0,\xi)| + |\dot{e_{+}}(t_0,\xi)||e_{-}(t,\xi)|) 
	\\
	&\cleq |\xi|^{-1}( \tau^{\frac{1}{2}+\re \nu}
	|\xi|\tau_0^{-\frac{1}{2} - \re \nu} 
	+|\xi|\tau_0^{-\frac{1}{2} + \re \nu} \tau^{\frac{1}{2} - \re \nu})
	\\
	&\ceq \tau^{\frac{1}{2}+\re \nu}
	\tau_0^{-\frac{1}{2} - \re \nu}
	+\tau_0^{-\frac{1}{2} + \re \nu} \tau^{\frac{1}{2} - \re \nu}
	\\
	&\cleq \tau^{\frac{1}{2}+\re \nu}
	\tau_0^{-\frac{1}{2} - \re \nu}
	\\
	&\ceq (1+t)^{\frac{1}{2}+\re \nu} (1+t_0)^{-\frac{1}{2} - \re \nu}
	\\
	&\cleq (1+t)^{\frac{1}{2}+\re \nu}  
	\\
	&\cleq |\xi|^{-\frac{1}{2}-\re \nu}.
\end{align*}

We also estimate $E_1$ as follows.
\begin{align*}
	|E_{1}(t,t_0,\xi)| 
	&\cleq |\xi|^{-1} (|e_{+}(t_0,\xi)||e_{-}(t,\xi)| + |e_{+}(t,\xi)||e_{-}(t_0,\xi)|) 
	\\
	&\cleq |\xi|^{-1} (  \tau_0^{\frac{1}{2} + \re \nu} \tau^{\frac{1}{2} - \re \nu}  + \tau^{\frac{1}{2} + \re \nu} \tau_0^{\frac{1}{2} - \re \nu}  )
	\\
	&\cleq |\xi|^{-1} \tau^{\frac{1}{2} + \re \nu} \tau_0^{\frac{1}{2} - \re \nu} 
	\\
	&\cleq 
	\begin{cases}
	|\xi|^{-1}\tau_0^{\frac{1}{2} - \re \nu} 
	\\
	(1+t)^{\frac{1}{2} + \re \nu} (1+t_0)^{\frac{1}{2} - \re \nu}
	\end{cases}
	\\
	& \ceq
	\begin{cases}
	 (1+t_0)^{\frac{1}{2} - \re \nu}  |\xi|^{-\frac{1}{2}- \re \nu}
	 \\
	(1+t_0)^{\frac{1}{2} - \re \nu}(1+t)^{\frac{1}{2} + \re \nu}
	\end{cases}.
\end{align*}

We also have 
\begin{align*}
	|\dot{E_{0}}(t,t_0,\xi)| 
	&\cleq |\xi|^{-1} (|\dot{e_{+}}(t,\xi)||\dot{e_{-}}(t_0,\xi)| + |\dot{e_{+}}(t_0,\xi)||\dot{e_{-}}(t,\xi)|) 
	\\
	&\cleq |\xi|(\tau^{-\frac{1}{2} + \re \nu}\tau_0^{-\frac{1}{2} - \re \nu}
	+ \tau_0^{-\frac{1}{2} + \re \nu}\tau^{-\frac{1}{2} - \re \nu} )
	\\
	& \cleq |\xi|\tau^{-\frac{1}{2} + \re \nu}\tau_0^{-\frac{1}{2} - \re \nu}
	\\
	&\cleq (1+t_0)^{-\frac{1}{2} - \re \nu} (1+t)^{-\frac{1}{2} + \re \nu}.
\end{align*}

Moreover, we have
\begin{align*}
	|\dot{E_{1}}(t,t_0,\xi)| 
	&\cleq |\xi|^{-1} (|e_{+}(t_0,\xi)||\dot{e_{-}}(t,\xi)| + |\dot{e_{+}}(t,\xi)||e_{-}(t_0,\xi)|) 
	\\
	&\cleq \tau^{-\frac{1}{2} - \re \nu} \tau_0^{\frac{1}{2}+\re \nu} + \tau^{-\frac{1}{2} + \re \nu} \tau_0^{\frac{1}{2}-\re \nu}
	\\
	&\cleq  \tau^{-\frac{1}{2} + \re \nu} \tau_0^{\frac{1}{2}-\re \nu}
	\\
	&\cleq 
	(1+t_0)^{\frac{1}{2}-\re \nu}(1+t)^{-\frac{1}{2}+\re \nu}.
\end{align*}

\subsection{Estimates in the case of $\mu=1/4$}
\label{appA.4}

We next consider the case of $\mu = 1/4$, i.e., $\nu = 0$. In this case, we have the logarithmic growth. 

\noindent{\bf Case 2-1.} $\tau \geq \tau_0 \geq N$.

In this case, similar estimates to Case1-1. are valid.   

\noindent{\bf Case 2-2.} $\tau \geq N \geq \tau_0$. 

We have
\begin{align*}
	|E_{0}(t,t_0,\xi)| 
	&\cleq |\xi|^{-1} (|e_{+}(t,\xi)||\dot{e_{-}}(t_0,\xi)| + |\dot{e_{+}}(t_0,\xi)||e_{-}(t,\xi)|) 
	\\
	&\cleq   \tau_0^{-\frac{1}{2}} (1+|\log \tau_0|)
	+  \tau_0^{-\frac{1}{2}}
	\\
	&\cleq   \tau_0^{-\frac{1}{2}}(1+|\log \tau_0|)
	\\
	&\ceq   (1+t_0)^{-\frac{1}{2}} |\xi|^{-\frac{1}{2}} ( 1+|\log ((1+t_0)|\xi|)| )
	\\
	&\cleq 
	\begin{cases}
	|\xi|^{-\frac{1}{2} -\varepsilon}
	\\
	(1+t_0)^{-\frac{1}{2}} (1+t)^{\frac{1}{2}} \left(  1+\log\frac{1+t}{1+t_0} \right)
	\end{cases} 
\end{align*}
for sufficiently small $\varepsilon>0$,
where we use the following lemma. 
\begin{lemma}
\label{lemA.7}
Let $t \geq t_0 \geq0$ and $(1+t)|\xi| \geq N \geq (1+t_0)|\xi|$ for some positive number $N$. We have
\begin{align*}
	1+|\log ( (1+t_0)|\xi| )|
	\cleq 
	\begin{cases}
	(1+\log(1+t_0)) |\xi|^{-\varepsilon}
	\\
	1+\log\frac{1+t}{1+t_0}
	\end{cases}
\end{align*}
for small $\varepsilon>0$. 
\end{lemma}

\begin{proof}
By the assumption, we have $|\xi|\leq N$. 
The first inequality comes from
\begin{align*}
	1+|\log ((1+t_0)|\xi|)|
	\cleq (1+\log(1+t_0)) (1+\left|\log|\xi|\right|)
	\cleq (1+\log(1+t_0)) |\xi|^{-\varepsilon}.
\end{align*}
It also holds that 
\begin{align*}
	1+|\log ((1+t_0)|\xi|)|
	&= 
	\begin{cases}
	1+\log((1+t_0)|\xi|) & \text{ if } (1+t_0)|\xi| \geq 1
	\\
	1-\log((1+t_0)|\xi|) & \text{ if } (1+t_0)|\xi| < 1
	\end{cases}
	\\
	&= 
	\begin{cases}
	1+\log((1+t_0)|\xi|) & \text{ if } (1+t_0)|\xi| \geq 1
	\\
	1+\log(((1+t_0)|\xi|)^{-1}) & \text{ if } (1+t_0)|\xi| < 1
	\end{cases}
	\\
	&\cleq 
	\begin{cases}
	1 & \text{ if } (1+t_0)|\xi| \geq 1
	\\
	1+\log\frac{1+t}{1+t_0} & \text{ if }  (1+t_0)|\xi| < 1
	\end{cases}
	\\
	&\cleq 1+\log\frac{1+t}{1+t_0}
\end{align*}
since $|\xi|^{-1} \cleq (1+t)$ by the assumption. 
\end{proof}

We also estimate $E_1$ as follows.
\begin{align*}
	|E_{1}(t,t_0,\xi)| 
	&\cleq |\xi|^{-1} (|e_{+}(t_0,\xi)||e_{-}(t,\xi)| + |e_{+}(t,\xi)||e_{-}(t_0,\xi)|) 
	\\
	&\cleq |\xi|^{-1} (|e_{+}(t_0,\xi)| + |e_{-}(t_0,\xi)|)
	\\
	&\cleq |\xi|^{-1} (  \tau_0^{\frac{1}{2}}  +  \tau_0^{\frac{1}{2}} (1+|\log \tau_0|)  )
	\\
	&\cleq |\xi|^{-1} \tau_0^{\frac{1}{2}}(1+|\log \tau_0|)
	\\
	&\cleq   
	\begin{cases}
	|\xi|^{-1} 
	\\
	(1+t_0)^{\frac{1}{2}} (1+t)^{\frac{1}{2}}\left(  1+\log\frac{1+t}{1+t_0} \right)
	\end{cases},
\end{align*}
where we use $\tau_0^{\frac{1}{2}}(1+|\log \tau_0|) \cleq 1$ for $\tau_0 \leq N$ and Lemma \ref{lemA.7}.
We also have 
\begin{align*}
	|\dot{E_{0}}(t,t_0,\xi)| 
	&\cleq |\xi|^{-1} (|\dot{e_{+}}(t,\xi)||\dot{e_{-}}(t_0,\xi)| + |\dot{e_{+}}(t_0,\xi)||\dot{e_{-}}(t,\xi)|) 
	\\
	&\cleq |\xi|(|\dot{e_{-}}(t_0,\xi)| + |\dot{e_{+}}(t_0,\xi)|)
	\\
	&\cleq |\xi|\tau_0^{-\frac{1}{2}} (1+|\log \tau_0|)
	+ |\xi| \tau_0^{-\frac{1}{2}} 
	\\
	&\cleq |\xi|\tau_0^{-\frac{1}{2}}(1+|\log \tau_0|)
	\\
	&\cleq (1+t_0)^{-\frac{1}{2}} |\xi|^{\frac{1}{2}}\left(  1+\log\frac{1+t}{1+t_0} \right)
	\\
	&\cleq (1+t_0)^{-\frac{1}{2}} (1+t)^{-\frac{1}{2}}\left(  1+\log\frac{1+t}{1+t_0} \right).
\end{align*}
Moreover, we have
\begin{align*}
	|\dot{E_{1}}(t,t_0,\xi)| 
	&\cleq |\xi|^{-1} (|e_{+}(t_0,\xi)||\dot{e_{-}}(t,\xi)| + |\dot{e_{+}}(t,\xi)||e_{-}(t_0,\xi)|) 
	\\
	&\cleq  \tau_0^{\frac{1}{2}} + \tau_0^{\frac{1}{2}}(1+|\log \tau_0|)
	\\
	&\cleq \tau_0^{\frac{1}{2}}(1+|\log \tau_0|)
	\\
	&\ceq 
	 (1+t_0)^{\frac{1}{2}}|\xi|^{\frac{1}{2}}\left(  1+\log\frac{1+t}{1+t_0} \right)
	\\
	&\cleq 
	(1+t_0)^{\frac{1}{2}}(1+t)^{-\frac{1}{2}}\left(  1+\log\frac{1+t}{1+t_0} \right)
	\\
	&\cleq 1,
\end{align*}
where we use $s^{-1/2}\log s \leq C$ for $s \geq 1$ in the last inequality.

\noindent{\bf Case 2-3.} $N \geq \tau \geq \tau_0$. 

By Lemma \ref{lemA.6}, we have
%
\begin{align*}
	|E_{0}(t,t_0,\xi)| 
	&\cleq |\xi|^{-1} (|e_{+}(t,\xi)\dot{e_{-}}(t_0,\xi)- \dot{e_{+}}(t_0,\xi) e_{-}(t,\xi)|) 
	\\
	&\cleq \tau^{\frac{1}{2}}\tau_{0}^{-\frac{1}{2}} \left(1+ \log \frac{\tau}{\tau_0}\right)
	\\
	&\cleq (1+t)^{\frac{1}{2}}(1+t_0)^{-\frac{1}{2}} \left(1+ \log \frac{1+t}{1+t_0}\right)
	\\
	&\cleq |\xi|^{-\frac{1}{2}-\varepsilon}
\end{align*}
for small $\varepsilon>0$.

We also estimate $E_1$ as follows.
\begin{align*}
	|E_{1}(t,t_0,\xi)| 
	&\cleq |\xi|^{-1} (|e_{+}(t_0,\xi)e_{-}(t,\xi) - e_{+}(t,\xi)e_{-}(t_0,\xi)|) 
	\\
	&\cleq |\xi|^{-1}\tau_0^{\frac{1}{2}} \tau^{\frac{1}{2}}  \left(1+ \log \frac{\tau}{\tau_0}\right)
	\\
	&\cleq 
	\begin{cases}
	|\xi|^{-1}
	\\
	(1+t)^{\frac{1}{2}} (1+t_0)^{\frac{1}{2}} \left(1+ \log \frac{1+t}{1+t_0}\right)
	\end{cases},
\end{align*}
where we use $\tau_0^{\frac{1}{2}} \tau^{\frac{1}{2}}  \left(1+ \log \frac{\tau}{\tau_0}\right) \cleq 1$ for $\tau_0 \leq \tau \leq N$. 
We also have 
\begin{align*}
	|\dot{E_{0}}(t,t_0,\xi)| 
	&\cleq |\xi|^{-1} (|\dot{e_{+}}(t,\xi)\dot{e_{-}}(t_0,\xi)-\dot{e_{+}}(t_0,\xi)\dot{e_{-}}(t,\xi)|) 
	\\
	&\cleq |\xi|\tau^{-\frac{1}{2}}\tau_0^{-\frac{1}{2}} \left(1+ \log \frac{\tau}{\tau_0}\right)
	\\
	&\cleq (1+t_0)^{-\frac{1}{2}} (1+t)^{-\frac{1}{2}}\left(1+ \log \frac{1+t}{1+t_0}\right).
\end{align*}
Moreover, we have
\begin{align*}
	|\dot{E_{1}}(t,t_0,\xi)| 
	&\cleq |\xi|^{-1} (|e_{+}(t_0,\xi) \dot{e_{-}}(t,\xi)-\dot{e_{+}}(t,\xi)e_{-}(t_0,\xi)|) 
	\\
	&\cleq \tau^{-\frac{1}{2} } \tau_0^{\frac{1}{2}} \left(1+ \log \frac{\tau}{\tau_0}\right)
	\\
	&\cleq 
	(1+t_0)^{\frac{1}{2}}(1+t)^{-\frac{1}{2}} \left(1+ \log \frac{1+t}{1+t_0}\right)
	\\
	&\cleq 1.
\end{align*}

Combining these estimates, we get the following.
\begin{lemma}
\label{lemA.8}
For any $t \geq t_0 \geq 0$ and any $\xi \in \mathbb{R}^d$, the following estimates are valid. 
If $\mu \neq 1/4$, i.e., $\nu\neq0$, then 
\begin{align*}
	|E_{0}(t,t_0,\xi)| 
	&\cleq 
	\begin{cases}
	1+|\xi|^{-\frac{1}{2}-\re\nu}
	\\
	(1+t)^{\frac{1}{2}+\re \nu} (1+t_0)^{-\frac{1}{2} - \re\nu}
	\end{cases},
	\\
	|E_{1}(t,t_0,\xi)| 
	&\cleq 
	\begin{cases}
	|\xi|^{-1} + (1+t_0)^{\frac{1}{2} - \re \nu} |\xi|^{-\frac{1}{2}-\re\nu}
	\\
	(1+t)^{\frac{1}{2}+\re\nu} (1+t_0)^{\frac{1}{2} - \re\nu}
	\end{cases},
	\\
	|\dot{E_{0}}(t,t_0,\xi)| &\cleq 
	\begin{cases}
	|\xi| +(1+t_0)^{-\frac{1}{2} - \re\nu}|\xi|^{\frac{1}{2}-\re\nu}
	\\
	|\xi| + (1+t)^{-\frac{1}{2}+\re\nu} (1+t_0)^{-\frac{1}{2} - \re\nu}
	\end{cases},
	\\
	|\dot{E_{1}}(t,t_0,\xi)| 
	&\cleq (1+t)^{-\frac{1}{2}+\re\nu} (1+t_0)^{\frac{1}{2} - \re\nu}.
\end{align*}
If $\mu =1/4$, i.e., $\nu = 0$, then
\begin{align*}
	|E_{0}(t,t_0,\xi)| 
	&\cleq 
	\begin{cases}
	1+|\xi|^{-\frac{1}{2}-\varepsilon}
	\\
	(1+t)^{\frac{1}{2}} (1+t_0)^{-\frac{1}{2}} \left(1+\log\frac{1+t}{1+t_0}\right)
	\end{cases},
	\\
	|E_{1}(t,t_0,\xi)| 
	&\cleq 
	\begin{cases}
	|\xi|^{-1}
	\\
	(1+t)^{\frac{1}{2}} (1+t_0)^{\frac{1}{2}}\left(1+\log\frac{1+t}{1+t_0}\right)
	\end{cases},
	\\
	|\dot{E_{0}}(t,t_0,\xi)| &\cleq 
	1+|\xi|,
	\\
	|\dot{E_{1}}(t,t_0,\xi)| 
	&\cleq 1
\end{align*}
for arbitrarily small $\varepsilon>0$. 
\end{lemma}

From Lemma \ref{lemA.8}, we obtain the energy estimates in Lemma \ref{lem2.1} when $\mu >0$, i.e., $\re \nu <1/2$. 

\begin{remark}
We find that the estimate of the $\dot{H}^1$-norm $\| v_l(t) \|_{\dot{H}^1}$ is independent of $t_0$ when $\re \nu <1/2$. In Lemma \ref{lemA.8}, for example, we have
\begin{align*}
	|E_{1}(t,t_0,\xi)| 
	\cleq  |\xi|^{-1} + (1+t_0)^{\frac{1}{2} - \re \nu} |\xi|^{-\frac{1}{2}-\re\nu}.
\end{align*}
It seems to depend on $t_0$. However, in the proof, $(1+t_0)^{\frac{1}{2} - \re \nu} |\xi|^{-\frac{1}{2}-\re\nu}$ appears only in the case of $(1+t_0)|\xi| \leq N$. In this case, $|\xi|(1+t_0)^{\frac{1}{2} - \re \nu} |\xi|^{-\frac{1}{2}-\re\nu} =(1+t_0)^{\frac{1}{2} - \re \nu} |\xi|^{\frac{1}{2}-\re\nu} \leq N^{\frac{1}{2} - \re \nu}$ since $\frac{1}{2} - \re \nu \geq 0$. Therefore, $|E_{1}(t,t_0,\xi)| \leq C$ and $C$ is independent of $t_0$. The similar independence can be checked for $E_1$, $\dot{E_0}$, and  $\dot{E_1}$ when $\re \nu <1/2$.
\end{remark}

Moreover, when $\mu < 0$, i.e., $\re \nu > 1/2$, the $\dot{H}^1$-norm of the linear solution may be unbounded as follows. 
\begin{corollary}
\label{corA.9}
Let $\mu \in (-\infty,0)$. Then, the solution $v_l$ of \eqref{KG} satisfies the following estimates.
\begin{align*}
	&\| v_l(t) \|_{\dot{H}^1} +\| \partial_t v_l(t) \|_{L^2} \cleq_{t_0} (1+t)^{-\frac{1}{2}+\re\nu} (\| v_{l,0} \|_{H^1} +\| v_{l,1} \|_{L^2}),
\end{align*}
where the implicit constant depends on $t_0$ and is independent of $t$. 
Moreover, we have the following estimate.
\begin{align*}
	&\| v_l(t) \|_{L^2} \cleq
	(1+t)^{\frac{1}{2}+\re\nu} (1+t_0)^{\frac{1}{2} - \re\nu} (\| v_{l,0} \|_{L^2}+\| v_{l,1} \|_{L^2})
\end{align*}
where the implicit constants are independent of $t_0$ and $t$. 
\end{corollary}

\begin{proof}
We note that $\re\nu >1/2$ when $\mu<0$. 
As in Appendix \ref{appA.4}, we have estimates such as
\begin{align*}
	|\xi||E_0(t,t_0,\xi)| 
	&\cleq |\xi|^{\frac{1}{2}-\re\nu} \cleq (1+t)^{-\frac{1}{2}+\re\nu} 
\end{align*}
when $\tau \geq N  \geq \tau_0$, and 
\begin{align*}
	|\xi||E_0(t,t_0,\xi)| 
	&\cleq |\xi| \tau^{\frac{1}{2}+\re\nu}  \tau_0^{-\frac{1}{2}-\re\nu}
	\cleq \tau |\xi| \tau^{-\frac{1}{2}+\re\nu}  \tau_0^{-\frac{1}{2}-\re\nu}
	\\
	&\cleq N |\xi| (1+t)^{-\frac{1}{2}+\re\nu} (1+t_0)^{-\frac{1}{2}-\re\nu} |\xi|^{-1}
	\cleq (1+t)^{-\frac{1}{2}+\re\nu}
\end{align*}
when $N\geq \tau \geq  \tau_0$. The similar estimates hold for $E_1$, $\dot{E_0}$, and $\dot{E_1}$ and thus the statement holds.
\end{proof}

\begin{remark}
Lemma \ref{lem2.1} and Corollary \ref{corA.9} imply the following estimates for the scale-invariant damping \eqref{DW} with $\mu_2=0$, which were obtained by Wirth \cite[Theorem 3.4]{Wir04}.
\begin{align*}
	&\norm{u_l(t)}_{\dot{H}^1} + \norm{\partial_t u_l(t)}_{L^2} \cleq 
	\begin{cases}
	(1+t)^{-\frac{\mu_1}{2}} & \text{ if }0< \mu_1 \leq 2,
	\\
	(1+t)^{-1} &  \text{ if } \mu_1>2,
	\end{cases}
	\\
	&\norm{u_l(t)}_{L^2} \cleq 
	\begin{cases}
	(1+t)^{1-\mu_1} & \text{ if }0< \mu_1<1,
	\\
	1+\log(1+t) & \text{ if } \mu_1 =1,
	\\
	1 &  \text{ if } \mu_1>1.
	\end{cases}
\end{align*}
We note that $\re \nu =\frac{1}{2}|\mu_1-1|$ when $\mu_2=0$ and we recall that $u_l=(1+t)^{-\mu_1/2}v_l$ and $\partial_t u_l =(-\mu_1/2)(1+t)^{-\mu_1/2-1}v_l+(1+t)^{-\mu_1/2} \partial_t v_l$. Therefore, Lemma \ref{lem2.1} and Corollary \ref{corA.9} are extension of the estimates by Wirth \cite[Theorem 3.4]{Wir04}. 
\end{remark}


\section{Solutions}
\label{appB}


\begin{definition}
\label{def}
We say that $v$ is a solution of \eqref{NLKG} on $I=[0,T)$ if $v \in C(I:H^1(\mathbb{R}^d))$, $v \in L^{p_1}(I:L^{2p_1}(\mathbb{R}^d))$, $(v(0),\partial_t v(0))=(v_0,v_1)$, and $v$ satisfies 
\begin{align*}
	v(t) = \mathcal{E}_{0}(t,0) v_0 + \mathcal{E}_{1}(t,0) v_1 + \int_{0}^{t} \mathcal{E}_{1}(t,s) \frac{\lambda}{(1+s)^{\frac{2\mu_1}{d-2}}} |v(s)|^{\frac{4}{d-2}} v(s) ds.
\end{align*}
\end{definition}

\begin{definition}
We say that $v$ is a distributional solution of \eqref{NLKG} on an interval $[0,T)$ if $v$ satisfies $\partial_t^2 v - \Delta v +\mu (1+t)^{-2} v = \lambda(1+t)^{-\frac{2\mu_1}{d-2}} |v|^{\frac{4}{d-2}}v$ in $\mathcal{D}'((0,T) \times \mathbb{R}^d)$. 
\end{definition}

\begin{definition}
We say that $v$ is a solution of the wave integral equation of \eqref{NLKG} on an interval $I \ni0$ if $v$ satisfies $(v,v_t) \in C(I:H^1 \times L^2)$, $v \in L^{p_1}(I:L^{2p_1})$, and
\begin{align}
	v(t) &= \mathcal{W}_0(t) v_0 + \mathcal{W}_1(t)v_1
	+ \int_{0}^{t} \mathcal{W}_1(t-s) \frac{-\mu}{(1+s)^2} v(s)ds 
	\\ \notag
	& + \int_{0}^{t} \mathcal{W}_1(t-s)  \frac{\lambda}{(1+s)^{\frac{2\mu_1}{d-2}}} |v(s)|^{\frac{4}{d-2}} v(s) ds,
	\\
	&(v(0),v_t(0))=(v_0,v_1).
\end{align}
\end{definition}

\begin{lemma}
Let $3 \leq d \leq5$. 
If $v$ is a global solution of \eqref{NLKG} in the sense of Definition \ref{def}, then $v$ is the distributional solution to \eqref{NLKG}.
\end{lemma}

\begin{proof}
We will prove that, for any $\phi \in C_0^{\infty}((0,T) \times \mathbb{R}^d)$, 
\begin{align}
\label{eqB.4}
	\int_{(0,T)\times \mathbb{R}^d} v \left( \partial_t^2 -\Delta + \frac{\mu}{(1+t)^2}\right) \phi dtdx
	=\int_{(0,T)\times \mathbb{R}^d}N(t,v) \phi dtdx
\end{align}
where we set $N(t,v):= \lambda(1+t)^{-\frac{2\mu_1}{d-2}} |v(t)|^{\frac{4}{d-2}} v(t)$. We may assume that $v$ is smooth by an approximation. (Indeed, let $(v_{0,n},v_{1,n}) \in C_{0}^{\infty}(\mathbb{R}^d)$ satisfy $\| v_{0,n} -v_0 \|_{H^1} + \| v_{1,n}-v_1 \|_{L^2} \to 0$ as $n \to \infty$ and $v_n$ be the solution of the integral equation starting from $(v_{0,n},v_{1,n})$. Then, $v_n$ is smooth and $\| v_{n} -v \|_{L_t^{\infty}(K:H^1)} + \| \partial_t v_{n}-\partial_t v \|_{L_t^{\infty}(K:L^2)} \to 0$ as $n \to \infty$ for an arbitrary fixed compact interval $K$.)
Setting 
\begin{align*}
	v_{L}:= \mathcal{E}_{0}(t,0) v_0 + \mathcal{E}_{1}(t,0) v_1
	\quad v_{N}:= \int_{0}^{t} \mathcal{E}_{1}(t,s)N(s,v) ds.
\end{align*}
It is trivial that
\begin{align*}
	\int_{(0,T)\times \mathbb{R}^d} v_{L} \left( \partial_t^2 -\Delta + \frac{\mu}{(1+t)^2}\right) \phi dtdx =0.
\end{align*}
Thus, it is enough to show that
\begin{align}
\label{eqB.5}
	\int_{(0,T)\times \mathbb{R}^d} v_{N} \left( \partial_t^2 -\Delta + \frac{\mu}{(1+t)^2}\right) \phi dtdx
	=\int_{(0,T)\times \mathbb{R}^d}N(t,v) \phi dtdx.
\end{align}
Now, we have
\begin{align*}
	\partial_t v_{N} 
	=\int_{0}^{t}  \partial_t (\mathcal{E}_{1}(t,s) N(s,v) ) ds
\end{align*}
and 
\begin{align*}
	\partial_t^2 v_{N} 
	&=N(t,v) + \int_{0}^{t}   \partial_t^2 (\mathcal{E}_{1}(t,s) N(s,v) )  ds 
\end{align*}
since $\mathcal{E}_{1}(t,t)=0$ and $ \dot{\mathcal{E}_{1}}(t,t)=1$. From the above calculation and integration by parts, we have
\begin{align*}
	&\int_{(0,T)\times \mathbb{R}^d} v_{N} \left( \partial_t^2 -\Delta + \frac{\mu}{(1+t)^2}\right) \phi dtdx
	\\
	&=\int_{(0,T)\times \mathbb{R}^d} N(t,v) \phi  dtdx
	+\int_{(0,T)\times \mathbb{R}^d}  \phi \int_{0}^{t} \left( \partial_t^2 -\Delta + \frac{\mu}{(1+t)^2}\right) G(t,s,x) ds dtdx,
\end{align*}
where we set $G(t,s,x):=\mathcal{E}_{1}(t,s) N(s,v)$, which is twice differentiable since $3\leq d \leq 5$, and thus the second term is zero. We get \eqref{eqB.5}. 
%
\end{proof}

\begin{lemma}
Assume that $v$ is a distributional solution of \eqref{NLKG} on an interval $I=[0,T)$ and $v \in C(I : H^1 \times L^2)$ and $v \in L^{p_1}(I:L^{2p_1})$. Then $v$ is also a solution of the wave integral equation of \eqref{NLKG} on $I$.
\end{lemma}

\begin{proof}
This follows from standard method (see \cite[Lemma 1.21 in p.10]{Ken15} for example). However, we give the proof. 

Set 
\begin{align*}
	V(t):= v(t) -  (\mathcal{W}_0(t) v_0 + \mathcal{W}_1(t)v_1).
\end{align*}
Then, $V(0)=V_t(0)=0$ and $V$ satisfies $\partial_t^2 V -\Delta V = - \frac{\mu}{(1+t)^2} v +\frac{\lambda}{(1+t)^{\frac{2\mu_1}{d-2}}} |v|^{\frac{4}{d-2}} v$ in $\mathcal{D}'((0,T) \times \mathbb{R}^d)$.

For $h \in C_{0}^{\infty}((0,\infty) \times \mathbb{R}^d)$, we define
\begin{align*}
	H(t,x) =  -\int_{t}^{\infty} \mathcal{W}_1(t-s) h(s) ds
\end{align*}
for $t \in \mathbb{R}$. Then, $H \in C^{\infty}(\mathbb{R}^{1+d})$ and $H=0$ for large $t$.  The support of $H(t)$ in $x$ is compact from the finite propagation.  Moreover, $(\partial_t^2 -\Delta) H =h$. 

Let $\varphi \in C^{\infty}(\mathbb{R})$ satisfy $\varphi(t)=1$ if $t \geq 1$ and $\varphi(t)=0$ if $t \leq 1/2$.  
We set
\begin{align*}
	H^a(t,x)=\varphi\left(\frac{t}{a}\right)H(t,x)
\end{align*}
for $a\in (0,1]$.
Then, $H^a \in C_{0}^{\infty}(\mathbb{R}^{1+d})$.
By the distributional equation, we have
\begin{align*}
	&\int_{(0,\infty) \times \mathbb{R}^d}  V (\partial_t^2 - \Delta) H^a dtdx 
	\\
	&\quad = \int_{(0,\infty) \times \mathbb{R}^d}  \left( - \frac{\mu}{(1+t)^2} v +\frac{\lambda}{(1+t)^{\frac{2\mu_1}{d-2}}} |v(t)|^{\frac{4}{d-2}} v(t)\right)H^a dtdx.
\end{align*}
Since $v\in L_t^{p_1}L_x^{2p_1}$, by the dominated convergence theorem, 
it holds that
\begin{align*}
	&\lim_{ a \to \infty} \int_{(0,\infty) \times \mathbb{R}^d} \left( - \frac{\mu}{(1+t)^2} v +\frac{\lambda}{(1+t)^{\frac{2\mu_1}{d-2}}} |v|^{\frac{4}{d-2}} v\right)H^a dtdx
	\\
	&=\int_{(0,\infty) \times \mathbb{R}^d} \left( - \frac{\mu}{(1+t)^2} v +\frac{\lambda}{(1+t)^{\frac{2\mu_1}{d-2}}} |v|^{\frac{4}{d-2}} v\right)H dtdx.
\end{align*}
By the Fubini theorem, it holds that
\begin{align*}
	&\int_{(0,\infty) \times \mathbb{R}^d}  \left( - \frac{\mu}{(1+t)^2} v +\frac{\lambda}{(1+t)^{\frac{2\mu_1}{d-2}}} |v|^{\frac{4}{d-2}} v\right)H dtdx
	\\
	&=-\int_{\mathbb{R}^d}  \int_{0}^{\infty} \int_{0}^{s}  \left( - \frac{\mu}{(1+t)^2} v +\frac{\lambda}{(1+t)^{\frac{2\mu_1}{d-2}}} |v|^{\frac{4}{d-2}} v\right) \mathcal{W}_1(t-s) h(s) dt dsdx
	\\
	&=\int_{\mathbb{R}^d}  \int_{0}^{\infty}\left\{  \int_{0}^{s} \mathcal{W}_1(s-t)  \left( - \frac{\mu}{(1+t)^2} v +\frac{\lambda}{(1+t)^{\frac{2\mu_1}{d-2}}} |v|^{\frac{4}{d-2}} v\right) dt\right\} h(s)  dsdx,
\end{align*}
where we also use the self-adjointness of $\mathcal{W}_1$. 
On the other hand, 
\begin{align*}
	&\int_{(0,\infty) \times \mathbb{R}^d}  V (\partial_t^2 - \Delta) H^a dtdx 
	\\
	&=
	\int_{(0,\infty) \times \mathbb{R}^d}  V \left\{ \frac{1}{a^2} \varphi''\left(\frac{t}{a}\right)H(t,x)
	+ \frac{2}{a} \varphi'\left(\frac{t}{a}\right)\partial_t H(t,x) +\varphi\left(\frac{t}{a}\right)h(t,x) \right\} dtdx.
\end{align*}
Now, we have
\begin{align*}
	&\lim_{a\to 0}\int_{(0,\infty) \times \mathbb{R}^d} V \frac{1}{a^2} \varphi''\left(\frac{t}{a}\right)H(t,x) dtdx=0,
	\\
	&\lim_{a\to 0} \int_{(0,\infty) \times \mathbb{R}^d}  V \frac{2}{a} \varphi'\left(\frac{t}{a}\right)\partial_t H(t,x) dtdx=0.
\end{align*}
Indeed, since $(V(t),\partial_tV(t)) \in H^1 \times L^2$ and $(V(0),\partial_t V(0))=(0,0)$, we have $\lim_{t\to 0} \| V(t) \|_{L^2}/t=0$. Therefore, let $\varepsilon>0$ and choose $a_0$ satisfying $\| V(t) \|_{L^2} \leq \varepsilon t$ for $t \in (0,a_0]$. Then, we have
\begin{align*}
	&\left| \int_{(0,\infty) \times \mathbb{R}^d}  V \frac{1}{a^2} \varphi''\left(\frac{t}{a}\right)H(t,x) dtdx \right|
	\\
	&=\left| \int_{0}^{a}\int_{\mathbb{R}^d}  V \frac{1}{a^2} \varphi''\left(\frac{t}{a}\right)H(t,x) dxdt \right|
	\\
	&\leq  \frac{1}{a^2} \int_{0}^{a}\int_{\mathbb{R}^d}  |V| |H(t,x)| dxdt
	\\
	&\leq  \frac{1}{a^2} \int_{0}^{a} \norm{V(t)}_{L_x^2}\norm{H(t)}_{L_x^2} dt
	\\
	&\leq \varepsilon \frac{1}{a^2} \int_{0}^{a} t dt  \norm{H}_{L_t^{\infty}L_x^2(0,a)}
	\\
	&\cleq \varepsilon.
\end{align*}
Thus, the first limit holds. The second limit is obtained similarly. By the dominated convergence theorem, we have
\begin{align*}
	\lim_{a\to 0} \int_{(0,\infty) \times \mathbb{R}^d} V \varphi\left(\frac{t}{a}\right)h(t,x) dtdx
	=\int_{(0,\infty) \times \mathbb{R}^d} Vh(t,x) dtdx.
\end{align*}
By the above argument, we obtain
\begin{align*}
	&\int_{(0,\infty) \times \mathbb{R}^d} Vh(t,x) dtdx
	\\
	& \quad =\int_{\mathbb{R}^d}  \int_{0}^{\infty}\left\{  \int_{0}^{t} \mathcal{W}_1(t-s)  \left( - \frac{\mu}{(1+s)^2} v +\frac{\lambda}{(1+s)^{\frac{2\mu_1}{d-2}}} |v|^{\frac{4}{d-2}} v\right) ds\right\} h(t)  dtdx.
\end{align*}
Since $h$ is arbitrary, this means that
\begin{align*}
	V(t)= \int_{0}^{t} \mathcal{W}_1(t-s)  \left( - \frac{\mu}{(1+s)^2} v +\frac{\lambda}{(1+s)^{\frac{2\mu_1}{d-2}}} |v|^{\frac{4}{d-2}} v\right)  ds
\end{align*}
and thus we obtain the desired integral equation
\begin{align*}
	v(t)= \mathcal{W}_0(t) v_0 + \mathcal{W}_1(t)v_1+ \int_{0}^{t} \mathcal{W}_1(t-s)  \left( - \frac{\mu}{(1+s)^2} v +\frac{\lambda}{(1+s)^{\frac{2\mu_1}{d-2}}} |v|^{\frac{4}{d-2}} v\right)  ds.
\end{align*}
\end{proof}

\begin{acknowledgement}
The first author is supported by by JSPS KAKENHI Grant-in-Aid for Early-Career Scientists JP18K13444 and the second author is supported by JSPS KAKENHI  Grant-in-Aid for Young Scientists (B) JP17K14218 and, partially, for Scientific Research (B) JP17H02854. The authors would like to express deep appreciation to Professor Yuta Wakasugi for introducing the known results.
\end{acknowledgement}


\end{document}